\documentclass[11pt,reqno]{amsart}
\usepackage[english]{babel}
\usepackage[a4paper,twoside]{geometry}
\usepackage{microtype}
\usepackage{enumerate}
\usepackage{csquotes}
\usepackage{longtable}
\usepackage{bbm}

\usepackage{xspace}
\patchcmd{\subsection}{-.5em}{.5em}{}{}
\usepackage[hyperfootnotes=false]{hyperref} 
\usepackage{color}
\hypersetup{
	colorlinks = true,
	linkcolor = {blue},
	urlcolor = {red},
	citecolor = {blue}
}

\makeatletter
\renewcommand{\tocsection}[3]{
  \indentlabel{\@ifnotempty{#2}{\ignorespaces#1 #2\quad}}\bfseries#3}
\renewcommand{\tocsubsection}[3]{
  \indentlabel{\@ifnotempty{#2}{\ignorespaces#1 #2\quad}}#3}

\newcommand\@dotsep{4.5}
\def\@tocline#1#2#3#4#5#6#7{\relax
  \ifnum #1>\c@tocdepth
  \else
    \par \addpenalty\@secpenalty\addvspace{#2}
    \begingroup \hyphenpenalty\@M
    \@ifempty{#4}{
      \@tempdima\csname r@tocindent\number#1\endcsname\relax
    }{
      \@tempdima#4\relax
    }
    \parindent\z@ \leftskip#3\relax \advance\leftskip\@tempdima\relax
    \rightskip\@pnumwidth plus1em \parfillskip-\@pnumwidth
    #5\leavevmode\hskip-\@tempdima{#6}\nobreak
    \leaders\hbox{$\m@th\mkern \@dotsep mu\hbox{.}\mkern \@dotsep mu$}\hfill
    \nobreak
    \hbox to\@pnumwidth{\@tocpagenum{\ifnum#1=1\bfseries\fi#7}}\par
    \nobreak
    \endgroup
  \fi}
\AtBeginDocument{
\expandafter\renewcommand\csname r@tocindent0\endcsname{0pt}
}
\def\l@subsection{\@tocline{2}{0pt}{2.5pc}{5pc}{}}
\makeatother

\usepackage{amsmath}
\usepackage{amsthm}
\usepackage{amssymb}
\usepackage{mathrsfs} 
\usepackage{eucal}
\usepackage{mathtools}

\usepackage{graphicx}
\usepackage{tikz}
\usetikzlibrary{cd}

\newcounter{results}[section]

\theoremstyle{plain}
\newtheorem{theorem}[results]{Theorem}

\newtheorem{lemma}[results]{Lemma}
\newtheorem{proposition}[results]{Proposition}
\newtheorem{corollary}[results]{Corollary}

\theoremstyle{remark}
\newtheorem{remark}[results]{Remark}

\theoremstyle{definition}
\newtheorem{definition}[results]{Definition}

\numberwithin{equation}{section}

\newcommand{\B}{\mathbb{B}}

\renewcommand{\H}{\mathbb{H}}
\newcommand{\M}{\mathbb{M}}
\newcommand{\N}{\mathbb{N}}

\newcommand{\R}{\mathbb{R}}

\renewcommand{\AA}{\mathscr{A}}

\newcommand{\EE}{\mathscr{E}}

\newcommand{\VV}{\mathscr{V}}
\newcommand{\WW}{\mathscr{W}}
\newcommand{\ii}{{\mbox{\boldmath$i$}}}
\newcommand{\mm}{{\mbox{\boldmath$m$}}}

\renewcommand{\ss}{{\mbox{\boldmath$s$}}}
\renewcommand{\tt}{{\mbox{\boldmath$t$}}}

\newcommand{\dD}{{\mbox{\boldmath$D$}}}

\newcommand{\gG}{{\mbox{\boldmath$G$}}}

\newcommand{\ggamma}{{\mbox{\boldmath$\gamma$}}}

\newcommand{\mmu}{{\mbox{\boldmath$\mu$}}}

\newcommand{\pphi}{{\boldsymbol \phi}}

\newcommand{\sfd}{{\sf d}}

\newcommand{\sfx}{{\sf x}}

\newcommand{\rmC}{{\mathrm C}}

\newcommand{\rmB}{{\mathrm B}}
\newcommand{\rmD}{{\mathrm D}}

\newcommand{\Kliminf}{K\kern-3pt-\kern-2pt\mathop{\rm lim\,inf}\limits}  % Kuratowski liminf di insiemi
\newcommand{\supp}{\mathop{\rm supp}\nolimits}   % supporto 
\newcommand{\Lip}{\mathop{\rm Lip}\nolimits}          %Lipschitz
\newcommand{\Lipb}{\mathop{\rm Lip}_b\nolimits}          %Lipschitz
\newcommand{\lip}{\mathop{\rm lip}\nolimits}          %Lipschitz
\renewcommand{\d}{{ \mathrm d}}

\newcommand{\restr}[1]{\lower3pt\hbox{$|_{#1}$}}

\newcommand{\Leb}[1]{{\mathscr L}^{#1}}      % Misura di Lebesgue
\newcommand{\la}{{\langle}}                  % brackets
\newcommand{\ra}{{\rangle}}
\newcommand{\down}{\downarrow}              %frecce in su e in giu nei limiti

\newcommand{\eps}{\varepsilon}  
\newcommand{\bambalau}{{\raise.3ex\hbox{$\chi$}}}
\newcommand{\weakto}{\rightharpoonup}
\newcommand{\prob}{\mathcal P}

\renewcommand{\mm}{\mathfrak m}

\newcommand{\J}I
\newcommand{\CE}{\mathsf{C\kern-1pt E}}
\newcommand{\NE}{\mathsf{N\kern-2.5pt E}}
\newcommand{\wCE}{\mathsf{wC\kern-1pt E}}

\newcommand{\pCE}{\mathsf{pC\kern-1pt E}}

\newcommand{\uphi}{\pphi}

\newcommand{\sqm}[1]{\mathsf m_p^p(#1)} % Square moment of a probability measure
\newcommand{\rsqm}[1]{\mathsf m_p(#1)} % Moment of a probability measure
\newcommand{\mres}{\mathbin{\vrule height 1.6ex depth 0pt width
0.13ex\vrule height 0.13ex depth 0pt width 1.3ex}}
\newcommand{\ccyl}[2]{\mathfrak C\big (#1,#2 \big)}
\newcommand{\domG}{\mathcal D}

\newcommand{\lin}[1]{\mathsf L_{#1}}
\newcommand{\prbt}{{\prob_p(\R^d)}}

\newcommand{\dist}{\mathsf{dist}}
\newcommand{\pd}{{p}}
\newcommand{\jj}{\mathcal{J}}

\title[The general class of Wasserstein Sobolev spaces]{The general class of Wasserstein Sobolev spaces: density of cylinder functions, reflexivity, uniform convexity and Clarkson's inequalities}

\author{Giacomo Enrico Sodini}
\address{Giacomo Enrico Sodini: Institut für Mathematik - Fakultät für Mathematik - Universität Wien, Oskar-Morgenstern-Platz 1, 1090 Wien (Austria)}
\email{giacomo.sodini@univie.ac.at}


\begin{document}

\begin{abstract} We show that the algebra of cylinder functions in the Wasserstein Sobolev space $H^{1,q}(\prob_p(X,\sfd), W_{p, \sfd}, \mm)$ generated by a finite and positive Borel measure $\mm$ on the $(p,\sfd)$-Wasserstein space $(\prob_p(X, \sfd), W_{p, \sfd})$ on a complete and separable metric space $(X,\sfd)$ is dense in energy. As an application, we prove that, in case the underlying metric space is a separable Banach space $\B$, then the Wasserstein Sobolev space is reflexive (resp.~uniformly convex) if $\B$ is reflexive (resp.~if the dual of $\B$ is uniformly convex). Finally, we also provide sufficient conditions for the validity of Clarkson's type inequalities in the Wasserstein Sobolev space.
\end{abstract}

\maketitle
\tableofcontents
\thispagestyle{empty}

\section{Introduction}
The study of Sobolev spaces on metric measure spaces is a well established area of interest in metric geometry, we refer to \cite{Heinonen-Koskela98,Koskela-MacManus98, pasqualetto, Savare22} for a general treatment of the subject. It is thus important to provide examples (as concrete as possible) of such Sobolev spaces built on relevant metric spaces; the study carried out in \cite{FSS22} goes precisely in this direction: the authors analyze the properties of the $2$-Sobolev space on the $2$-Wasserstein space \cite{Villani03, Villani09, AGS08, santambrogio} of probability measures on a separable Hilbert space $\H$.

 Let us also mention that, besides being interesting from a purely theoretical point of view, the study of Sobolev spaces on spaces of probability measures has also important applications to functional analysis over spaces of probability measures. Indeed, problems such as evolutionary games or Kolmogorov equations in nonlinear filtering have been treated so far by classical notions of solutions, lacking weak formulations. In this regard, providing a functional analytic framework in which to set those problems may be very relevant. We refer to the introduction of \cite{FSS22} for a more detailed list of applications and references.  

In this work we generalize the results in \cite{FSS22} considering a general complete and separable metric space $(X,\sfd)$, instead of a Hilbert space, and we treat general exponents $p,q$ both in the order of the Wasserstein distance and in the Sobolev space (instead of $p=q=2$). Before entering into the details of the present work, let us briefly recall the definition of metric Sobolev space and the results of \cite{FSS22} that are the starting point for our analysis. 

Among the possible approaches to Sobolev spaces on metric  measure  spaces (for example the Newtonian one \cite{Shanmugalingam00, Bjorn-Bjorn11}), here we focus on the approach, strictly related to the ideas of Cheeger \cite{Cheeger99}, that is contained in the work of Ambrosio, Gigli and Savaré \cite{AGS14I} where they define the following concept of $q$-relaxed ($q \in (1,+\infty))$ gradient: given a metric measure space $(X, \sfd, \mm)$, we say that $G \in L^q(X, \mm)$ is a $q$-relaxed gradient of $f \in L^0(X, \mm)$ if there exist a sequence $(f_n)_n \subset \Lipb(X,\sfd)$ and $\tilde{G} \in L^q(X, \mm)$ such that
\begin{enumerate}
    \item $f_n \to f$ in $L^0(X, \mm)$ and $\lip_\sfd f \weakto \tilde{G}$ in $L^q(X, \mm)$,
    \item $\tilde{G} \le G$ $\mm$-a.e.~in $X$,
\end{enumerate}
where, for $f \in \Lipb(X, \sfd)$, the asymptotic Lipschitz constant of $f$ is defined as
\begin{equation}\label{eq:aslipintro}
  \lip_\sfd f(x):=
  \limsup_{y,z\to x,\ y\neq z}\frac{|f(y)-f(z)|}{\sfd(y,z)}, \quad x \in X.
\end{equation}
The $q$-Cheeger energy of $f \in L^0(X, \mm)$ is then defined as
\[ \CE_q(f) := \int_X |\rmD f|_{\star, q}^q \d \mm,\]
where $|\rmD f|_{\star, q}$ is the minimal relaxed gradient of $f$ i.e.~the element of minimal norm in the set of relaxed gradients of $f$. It is well known that the $q$-Cheeger energy can be also characterized as the relaxation of the so called pre-$q$-Cheeger energy 
\[ \pCE_q(f):= \int_X (\lip_\sfd f)^q \d \mm, \quad f \in \Lipb(X, \sfd),\]
in the sense that
\begin{equation}\label{eq:relaxintroche}
  \CE_{q}(f) = \inf \left \{ \liminf_{n \to + \infty}
    \pCE_{q}(f_n) : (f_n)_n \subset \Lipb(X, \sfd), \, f_n \to f \text{ in } L^0(X,\mm)
  \right \}.
\end{equation}
The Sobolev space \emph{à la Cheeger} $H^{1,q}(X, \sfd, \mm)$ is thus the vector space of functions $f\in L^q(X, \mm)$ with finite Cheeger energy endowed with the norm 
\[ |f|^q_{H^{1,q}(X, \sfd, \mm)}:= \int_X |f|^q \d \mm + \CE_q(f)\]
which makes it a Banach space. A remarkable result \cite{AGS14I} is the so called strong approximation property by Lipschitz functions: if $f \in H^{1,q}(X,\sfd, \mm)$ then there exists a sequence $(f_n)_n \subset \Lipb(X,\sfd)$ such that
\begin{equation}\label{eq:sap}
    f_n \to f, \quad \lip_\sfd f_n \to |\rmD f|_{\star, q} \text{ in } L^q(X, \mm).
\end{equation}
In \cite{FSS22} the authors provide a general criterion for the validity of the same property where, instead of the whole $\Lipb(X,\sfd)$, the approximating sequence can be taken from a suitable subalgebra $\AA \subset \Lipb(X,\sfd)$ satisfying
\[\mathbbm{1}\in \AA,\quad
  \text{for every $x_0,x_1\in X$ there exists $f\in \AA$:\quad $f(x_0)\neq f(x_1)$}, 
\]
where $\mathbbm{1}:X \to \R$ is the constant function equal to $1$. This is equivalent to say that $\AA$ is \textit{dense in energy in the Sobolev space}. In particular (see Theorem \ref{theo:startingpoint} below or \cite[Theorem 2.12]{FSS22} for the proof) it is proven that an equivalent condition is that for every $y \in X$ it holds
\begin{equation}\label{eq:suffcond}
|\rmD \sfd_y |_{\star, q, \AA} \le 1 \quad \mm\text{-a.e. in } X,   
\end{equation}
where $\sfd_y(x):= \sfd(x,y)$ and $|\rmD \cdot|_{\star,q ,\AA}$ is a suitably $\AA$-adapted notion of minimal relaxed gradient (cf.~Definition \ref{def:relgrad}). The second part of \cite{FSS22} is devoted to apply the criterion in \eqref{eq:suffcond} to the $L^2$-Kantorovich-Rubinstein-Wasserstein (in brief, Wasserstein) space on a separable Hilbert space $\mathbb{H}$, denoted by $(\prob_2(\mathbb{H}), W_2)$, with the algebra $\ccyl{\prob(\mathbb{H})}{\rmC_b^1(\mathbb{H})}$ of cylinder functions (cf.~Definition \ref{def:cyl}), which is the algebra of functions of the form
\begin{equation}\label{eq:acyl}
F(\mu) = \psi \left ( \int_{\mathbb{H}} \phi_1 \d \mu, \dots, \int_{\mathbb{H}} \phi_N \d \mu \right ), \quad \mu \in \prob(\mathbb{H}),  
\end{equation}
where $\psi \in \rmC_b^1(\R^N)$ and $\phi_n \in \rmC_b^1(\mathbb{H})$ for $n=1, \dots, N \in \N$. The density of $\ccyl{\prob(\mathbb{H})}{\rmC_b^1(\mathbb{H})}$ in $H^{1,2}(\prob_2(\mathbb{H}), W_2, \mm)$ (here $\mm$ is a finite and positive Borel measure on $\prob_2(\mathbb{H})$) is particularly interesting since cylinder functions come already with some structure; in particular to any function $F$ as in \eqref{eq:acyl} we can associate a notion of gradient given by
\begin{equation}\label{eq:diffintro}
\rmD F(\mu, x):= \sum_{n=1}^N \partial_n \psi \left ( \int_{\mathbb{H}} \phi_1 \d \mu, \dots, \int_{\mathbb{H}} \phi_N \d \mu \right ) \nabla \phi_n(x), \quad (\mu, x) \in \prob(\mathbb{H})\times \mathbb{H}.
\end{equation}
It is proven in \cite[Proposition 4.9]{FSS22} that, for every $F \in \ccyl{\prob(\mathbb{H})}{\rmC_b^1(\mathbb{H})}$, we have
\begin{equation}\label{eq:pcera}
\pCE_2(F) = \int_{\prob_2(\mathbb{H})\times \mathbb{H}} | \rmD F(\mu, x)|^2 \d \mu (x) \d \mm(\mu),   
\end{equation}
so that the pre-Cheeger energy satisfies the parallelogram identity (cf.~\cite[Subsection 4.2]{FSS22}) and thus forces the Cheeger energy to be a quadratic form. This amounts to say that the Sobolev space $H^{1,2}(\prob_2(\mathbb{H}), W_2, \mm)$ is a Hilbert space, a crucial property in the theory of metric Sobolev spaces \cite{Lott-Villani07,Sturm06I, Sturm06II, AGS14I, Gigli15-new}.
\medskip
\quad \\
The aim of this work is, following the approach of \cite{FSS22}, to study the properties of the more general class of Wasserstein Sobolev spaces $H^{1,q}(\prob_p(\B, \sfd_{\|\cdot\|}), W_{p, \sfd_{\|\cdot\|}}, \mm)$ ($p,q \in (1,+\infty)$), where $(\B,\|\cdot\|)$ is a separable Banach space, $(\prob_p(\B, \sfd_{\|\cdot\|}), W_{p, \sfd_{\|\cdot\|}})$ is the $(p,\sfd_{\|\cdot\|})$-Wasserstein space on it and $\mm$ is positive and finite Borel measure on $\prob_p(\B, \sfd_{\|\cdot\|})$. First of all, in Proposition \ref{prop:equality} we provide a generalization of the representation in \eqref{eq:pcera}: if $F$ is a cylinder function in $\prob(\B)$ as in \eqref{eq:acyl}, then 
\begin{equation}\label{eq:mixpce}
\pCE_q(F) = \int_{\prob_p(\B, \sfd_{\|\cdot\|})} \left ( \int_{\B} \|\rmD F(\mu, x)\|_*^{p'} \d \mu(x) \right )^{q/p'} \d \mm (\mu),
\end{equation}
where $p'= p/(p-1)$ is the conjugate exponent of $p$, $\|\cdot\|_*$ is the dual norm in $\B^*$ and $\rmD F$ is adapted from \eqref{eq:diffintro} in the obvious way. Notice that, differently from the representation \eqref{eq:pcera}, the right hand side of \eqref{eq:mixpce} doesn't coincide with the $q$-th power of a $L^q$ norm on a suitable space but it is rather the $q$-th power of the norm in the $L^q$-direct integral of a family of Banach spaces (see Subsection \ref{sec:refl} for more details).

Approaching first the case in which $\B=\R^d$ and $\|\cdot\|$ is a sufficiently regular norm on $\R^d$ (cf.\eqref{eq:costf}), we are able to prove the density of $\ccyl{\prob(\R^d)}{\rmC_b^1(\R^d)}$ in the corresponding Sobolev space (first part of Theorem \ref{thm:main}) adapting the techniques of \cite{FSS22} (in particular using \eqref{eq:suffcond} and the representation in \eqref{eq:mixpce}) to this more general situation: this requires more sophisticated results (which may be interesting by themselves) in terms of the properties of Kantorovich potentials which are no longer convex if $p \ne 2$ (see Section \ref{subsec:Kuseful}).

The density of cylinder functions is then extended to the Sobolev-Wasserstien space corresponding to an arbitrary norm on $\R^d$ by an approximation procedure and then to an arbitrary separable Banach space $(\B, \|\cdot\|)$ via a standard embedding technique in $\ell^\infty(\N)$ and finite dimensional projections (second part of Theorem \ref{thm:main} and Corollary \ref{cor:densban}, respectively). The precise statement of the result is the following.
\begin{theorem}\label{maintheo} Let $(\B, \|\cdot\|)$ be a separable Banach space; then the algebra of cylinder functions $\ccyl{\prob(\B)}{\rmC_b^1(\B)}$ is dense in $q$-energy in the Sobolev space $H^{1,q}(\prob_p(\B, \sfd_{\|\cdot\|}), W_{p, \sfd_{\|\cdot\|}}, \mm)$.
\end{theorem}
Let us also mention that we are able to extend the density result (Theorem \ref{thm:main3}) to an arbitrary complete and separable metric space $(X,\sfd)$ where, instead of using $\rmC_b^1$ functions (which of course are not available) to generate cylinder functions on $\prob(X)$, we consider a sequence $(\phi_k)_k \subset \Lipb(X, \sfd)$ such that
\[ \sfd(x,y)= \sup_k |\phi_k(x)-\phi_k(y)|, \quad \text{ for every } x,y \in X,\]
and we use the smallest unital algebra containing $(\phi_k)_k$.

In the same spirit of the Hilbertianity result in \cite{FSS22}, we are led to study which properties of the Banach space $\B$ are transferred to the Sobolev space, thanks to the density of cylinder functions provided by the above Theorem \ref{maintheo}. This is the case of uniform convexity and the validity of some Clarkson's type inequalities. The argument we adopt for such properties is similar: since the pre-Cheeger energy of a cylinder function as in \eqref{eq:mixpce} corresponds to the $q$-th power of a norm in a suitable Banach space, it is sufficient to prove that such norm enjoys the uniform convexity (resp.~the validity of Clarkson's type inequalities) to obtain that the pre-Cheeger satisfies the same property; thanks to the density of cylinder functions such a property is extended to the Cheeger energy and thus to the whole Sobolev norm (Theorems \ref{thm:uc} and \ref{cor:ilprimo}).
For what concerns reflexivity the argument is somehow standard: again using the representation of the pre-Cheeger energy in \eqref{eq:mixpce}, we can isometrically embed the Sobolev space into a reflexive Banach space (Theorem \ref{thm:refl}). The precise statement regarding reflexivity and uniform convexity is reported below (we refer to Subsection \ref{sec:clar} for the results related to Clarkson's type inequalities).
\begin{theorem} Let $(\B, \|\cdot\|)$ be a separable Banach space. If $\B$ is reflexive (resp.~its dual is uniformly convex), then the Sobolev space $H^{1,q}(\prob_p(\B, \sfd_{\|\cdot\|}), W_{p, \sfd_{\|\cdot\|}}, \mm)$ is reflexive (resp.~uniformly convex).
\end{theorem}
\medskip
\paragraph{\em\bfseries Plan of the paper} In \textbf{Section \ref{sec:main1}} we summarize the construction of metric Sobolev spaces depending on a subalgebra of  Lipschitz  and bounded functions $\AA$ and we report a few results of \cite{FSS22} concerning the density in energy of $\AA$. The beginning of \textbf{Section \ref{sec:Wasserstein}} contains the general framework for Wasserstein spaces we are going to work with and Subsection \ref{subsec:Kuseful} presents some compactness results for  Kantorovich  potentials in a specific geometric situation: these results combine the ideas of \cite{Gigli-Figalli,Gangbo-McCann96} to provide useful Lipschitz estimates on the potentials. \textbf{Section \ref{sec:main2}} contains the core of our density results: after stating the definition of cylinder functions in the framework of a metric space $(X,\sfd)$, in Subsection \ref{subsec:cylindrcial}, we study the asymptotic Lipschitz constant of a cylinder function (Proposition \ref{prop:equality}); in Subsection \ref{subsec:wsspace} we prove the density result when the base space is $\R^d$ endowed with an arbitrary norm (Theorem \ref{thm:main}); in Subsection \ref{sec:densxd} we generalize this result to a complete and separable metric space $(X,\sfd)$ (Theorem \ref{thm:main3}). Finally, in Section \ref{sec:refl} we prove that relevant properties of the underlying Banach space pass to the Sobolev space: first we treat reflexivity and uniform convexity (Theorems \ref{thm:refl}, \ref{thm:uc}) and then we study Clarkson's type inequalities in Subsection \ref{sec:clar} (see in particular Theorem \ref{cor:ilprimo}).

\medskip
\paragraph{\em\bfseries Acknowledgments} The author warmly thanks G.~Savaré for proposing the problem and for his many valuable suggestions.  The author is grateful to E.~Pasqualetto for his comments on a first draft of the present paper.  The author also gratefully acknowledges the support of the Institute for Advanced Study of the Technical University of Munich, funded by the German Excellence Initiative. The author was also partially financed by the Austrian Science Fund (FWF) project F65.  The authors is grateful to the anonymous reviewer for their valuable comments.

\section{Metric Sobolev spaces and density of subalgebras}
\label{sec:main1}
In this section we describe the general metric setting and we list a few results of \cite{FSS22} which are the starting point of our analysis. For this whole section, we fix a separable metric space $(X,\sfd)$, a finite and positive Borel measure $\mm$ on $(X,\sfd)$ (the triple $(X, \sfd, \mm)$ is called a Polish metric-measure space), an exponent $q \in (1,+\infty)$ and a unital and separating subalgebra $\AA \subset \Lip_b(X, \sfd)$ i.e.~satisfying 
\begin{equation}
  \label{eq:116}
  \mathbbm{1}\in \AA,\quad
  \text{for every $x_0,x_1\in X$ there exists $f\in \AA$:\quad $f(x_0)\neq f(x_1)$}, 
 \end{equation}
 where $\Lip_b(X,\sfd)$ is the space of real valued and bounded $\sfd$-Lipschitz functions on $X$ and $\mathbbm{1}:X \to \R$ is the constant function equal to $1$.\\
If $(Y, \sfd_Y)$ is another complete and separable metric space and $f:X \to Y$ is a Borel function, we define the finite (with the same total mass of $\mm$) and Borel measure $f_\sharp \mm$  on $(Y, \sfd_Y)$ as
\begin{equation}\label{eq:push} f_{\sharp}\mm(B) = \mm(f^{-1}(B)) \quad \text{ for every Borel set } B \subset Y.
\end{equation}
 We define the $\sfd$-asymptotic Lipschitz constant of $f\in \Lip_b(X, \sfd)$ as
\begin{equation}
  \label{eq:1}
  \lip_\sfd f(x):=\lim_{r\down0}\Lip(f,\rmB_{\sfd}(x,r),\sfd)=
  \limsup_{y,z\to x,\ y\neq z}\frac{|f(y)-f(z)|}{\sfd(y,z)}, \quad x \in X,
\end{equation}
where $\rmB_{\sfd}(x,r) \subset X$ denotes the $\sfd$-open ball centered at $x$ with  radius  $r>0$ and,
for $A \subset X$, the quantity $\Lip(f,A,\sfd)$ is defined as
\[ \Lip(f,A,\sfd):= \sup_{x,y\in A, \, x \neq y}\frac{|f(x)-f(y)|}{\sfd(x,y)}. \]
We denote by $L^0(X, \mm)$ the space of real valued and Borel measurable functions on $X$, identified up to equality $\mm$-a.e.~and, analogously, by $L^r(X,\mm)$ the usual Lebesgue spaces of real valued, $r$-summable and Borel measurable functions, identified up to equality $\mm$-a.e., $r\in [1,+\infty]$. We endow $L^0(X,\mm)$ with the topology of the convergence in $\mm$-measure, while $L^r(X, \mm)$ is endowed with the usual norm, $r \in [1, +\infty]$. When dealing with vector-valued or extended real-valued functions we also specify the codomain in the notation for $L^r$-spaces, i.e.~we write $L^r(X,\mu; Y)$, where $Y$ is a (subset of a) Banach space. The following is the definition of relaxed gradient we adopt \cite{AGS14I,AGS13,Savare22} (see also \cite{Shanmugalingam00, Bjorn-Bjorn11} for a different approach).
\begin{definition}[$(q,\AA)$-relaxed gradient]
  \label{def:relgrad}
  We say that $G\in L^q(X,\mm)$ is a $(q,\AA)$-relaxed gradient of $f\in L^0(X,\mm)$ if
  there exist a sequence $(f_n)_{n\in \N}\in \AA$ and $\tilde{G} \in L^q(X, \mm)$ such that:
  \begin{enumerate}
  \item $f_n\to f$ in $\mm$-measure and
    $\lip_\sfd f_n\to \tilde G$ weakly in $L^q(X,\mm)$;
  \item $\tilde G\le G$ $\mm$-a.e.~in $X$.  
  \end{enumerate}
\end{definition}
The next result is a simple but important property of relaxed gradients \cite{AGS14I,AGS13,Savare22}.
\begin{theorem}\label{theo:mrgexists} If
    $f \in L^0(X,\mm)$  has a $(q,\AA)$-relaxed gradient then there exists a unique element of minimal $L^q(X,\mm)$-norm in 
    \begin{displaymath}
      S(f):=\Big\{G \in L^q(X,\mm):
      \text{$G$ is a $(q,\AA)$-relaxed gradient of $f$}\Big\}.
    \end{displaymath}
\end{theorem}
Thanks to Theorem \ref{theo:mrgexists} the next definition is well posed.
\begin{definition}[Minimal relaxed gradient] Let $f \in L^0(X,\mm)$ be such that $f$ has a $(q,\AA)$-relaxed gradient. Its relaxed gradient with minimal $L^q(X,\mm)$-norm is denoted by $|\rmD f|_{\star,q,\AA}$ and called minimal $(q,\AA)$-relaxed gradient of $f$.
\end{definition}

\begin{definition}[Cheeger energy and Sobolev space]
   We call $D^{1,q}(X,\sfd,\mm;\AA)$ the set of functions in  $L^0(X,\mm)$ with
  a $(q,\AA)$-relaxed gradient and we define the $(q,\AA)$-Cheeger energy as
  \begin{equation}
    \label{eq:3}
    \CE_{q,\AA}(f):=\int_X |\rmD f|_{\star,q,\AA}^q(x)\,\d\mm(x)\quad
    \text{for every $f\in
  D^{1,q}(X,\sfd,\mm;\AA)$},
  \end{equation}
  with $\CE_{q,\AA}(f):=+\infty$ if $f \in L^0(X, \mm) \setminus D^{1,q}(X,\sfd,\mm;\AA).$
  The Sobolev space $H^{1,q}(X,\sfd,\mm;\AA)$ is defined as
  $L^q(X,\mm)\cap D^{1,q}(X,\sfd,\mm;\AA)$. The Sobolev norm of $f \in H^{1,q}(X,\sfd,\mm;\AA)$ is defined as
  \[ \|f\|_{H^{1,q}(X,\sfd,\mm;\AA)}^q:=\|f\|_{L^q(X, \mm)}^q+\CE_{q, \AA}(f).\]
\end{definition}

In the next theorem we collect a few properties of relaxed gradients and Sobolev spaces that will be useful (for a more comprehensive list and references to the proofs, see \cite{FSS22}).

\begin{theorem}
  \label{thm:omnibus}\ 
  \begin{enumerate}[\rm (1)]
  \item The set
    \begin{displaymath}
      S:=\Big\{(f,G)\in L^0(X,\mm)\times L^q(X,\mm):
      \text{$G$ is a $(q,\AA)$-relaxed gradient of $f$}\Big\}
    \end{displaymath}
    is convex and it is closed with respect to
    to the product topology of the convergence in $\mm$-measure and
    the weak convergence in $L^q(X,\mm)$.
    In particular, the restriction $S_r:=S\cap L^r(X,\mm)\times
    L^q(X,\mm)$ is weakly closed in
    $L^r(X,\mm)\times
    L^q(X,\mm)$ for every $r\in (1,+\infty)$.
  \item (Strong approximation) If
   $f \in D^{1,q}(X, \sfd,\mm; \AA)$ takes values in a closed
    (possibly unbounded) 
    interval $I\subset \R$ then
    there exists a sequence $f_n\in \AA$ with values in $I$
    such that
    \begin{equation}
      \label{eq:4}
      f_n \to f\text{ $\mm$-a.e.~in $X$},\quad
      \lip_\sfd f_n\to |\rmD f|_{\star,q,\AA}\text{ strongly in }L^q(X,\mm).
    \end{equation}
    If moreover $f\in L^r(X,\mm)$ for some $r\in [1,+\infty)$ then 
    we can also find a sequence as in \eqref{eq:4} converging strongly
    to $f$ in $L^r(X,\mm)$.
  \item (Pointwise minimality)
    If $G$ is a $(q,\AA)$-relaxed gradient of
    $f \in L^0(X,\mm)$  then $|\rmD
    f|_{\star,q,\AA}\le G$ $\mm$-a.e.~in $X$.
    \item (Truncations) If $f_j \in D^{1,q}(X, \sfd,\mm; \AA)$, $j=1,\cdots,J$,
      then
      also the functions $f_+:=\max(f_1,\cdots,f_J)$ and
      $f_-:=\min(f_1,\cdots,f_J)$ have $(q,\AA)$-relaxed gradient
      and
      \begin{align}
        \label{eq:9}
        |\rmD f_+|_{\star,q,\AA}=|\rmD f_j|_{\star,q,\AA}&\quad\text{$\mm$-a.e.~on }
                                                       \{x\in
                                                       X:f_+=f_j\},\\
        |\rmD f_-|_{\star,q,\AA}=|\rmD f_j|_{\star,q,\AA}&\quad\text{$\mm$-a.e.~on }
                                                       \{x\in X:f_-=f_j\}.
      \end{align}
      \item (Sobolev norm) $(H^{1,q}(X,\sfd,\mm;\AA),\|\cdot\|_{H^{1,q}(X,\sfd,\mm;\AA)}) $ is a Banach space.
  \end{enumerate}
\end{theorem}
Notice that, if $\AA=\Lip_b(X,\sfd)$, we will simply use the notations $|\rmD \cdot |_{\star, q}$, $D^{1,q}(X, \sfd, \mm)$, $\CE_{q}(\cdot)$, $H^{1,q}(X, \sfd, \mm)$ and $\|\cdot\|_{H^{1,q}(X, \sfd, \mm)}$ omitting the dependence on $\AA$.

\begin{remark}[Pre-Cheeger energy and its relaxation] \label{rem:relprec}
It is well known (see e.g.~\cite[Corollary 3.1.7]{Savare22}) that for every $r \in \{0\} \cup [1, +\infty)$ it holds that 
\begin{equation}\label{eq:relpreq}
  \CE_{q,\AA}(f) = \inf \left \{ \liminf_{n \to + \infty}
    \pCE_{q}(f_n) : f_n \in \AA, \, f_n \to f \text{ in } L^r(X,\mm)
  \right \},
    \quad f\in L^r(X,\mm),
\end{equation}
where the pre-Cheeger energy $\pCE_q: \Lipb(X, \sfd) \to [0, +\infty)$ is defined as
\begin{equation}\label{eq:prec} \pCE_{q}(f) := \int_X (\lip_\sfd f)^q \,\d \mm, \quad f \in \Lip_b(X,\sfd).
\end{equation}
\end{remark}

The main property we are interested in investigating is the density of the subalgebra $\AA$ in the metric Sobolev space.
\begin{definition}[Density in energy of a subalgebra of Lipschitz functions]
  \label{def:density}
  We say that $\AA\subset \Lip_b(X, \sfd)$ is \emph{dense in $q$-energy} in $D^{1,q}(X, \sfd, \mm)$ if for every $f\in D^{1,q}(X, \sfd, \mm)$ there exists a sequence $(f_n)_{n\in \N}$ satisfying
  \begin{equation}
    \label{eq:4bis}
    f_n\in \AA,\quad
    f_n \to f\text{ $\mm$-a.e.~in $X$},\quad
    \lip_\sfd f_n\to |\rmD f|_{\star,q}\text{ strongly in }L^q(X,\mm).
  \end{equation}
\end{definition}
\begin{remark}\label{rem:dense2} It is not difficult to see that Definition \ref{def:density} is equivalent to the equality $D^{1,q}(X,\sfd,\mm;\AA)=D^{1,q}(X,\sfd,\mm)$ with equal minimal relaxed gradients, the equality of the Sobolev spaces $H^{1,q}(X,\sfd,\mm;\AA)=H^{1,q}(X,\sfd,\mm)$ with equal norms and the following strong approximation property: for every $f \in H^{1,q}(X,\sfd,\mm)$ there exists a sequence $(f_n)_{n\in \N}$ satisfying
  \begin{equation}
    \label{eq:4tris}
    f_n\in \AA,\quad
    f_n \to f\text{ in $L^q(X,\mm)$},\quad
    \lip_\sfd f_n\to |\rmD f|_{\star,q}\text{ strongly in }L^q(X,\mm).
  \end{equation}
\end{remark}
The following characterization of the density of a subalgebra of Lipschitz and bounded functions is proven in \cite{FSS22}.
\begin{theorem}\label{theo:startingpoint}
Let $Y\subset X$ be a dense subset. Then
\begin{equation}\label{eq:214-15}
  \text{for every } y \in { Y}
  \text{ it holds}\quad
  \sfd_y\in D^{1,q}(X, \sfd, \mm; \AA),\quad
  \big|\rmD \sfd_y\big|_{\star,q,\AA}\le 1
\end{equation}
if and only if $\AA$ is dense in $q$-energy according to Definition
\ref{def:density}, where the function $\sfd_y : X \to [0, + \infty)$ is defined as 
\[ \sfd_y(x) := \sfd(x,y), \quad x \in X.\]
\end{theorem}

\section{Wasserstein spaces}
\label{sec:Wasserstein}
We devote the first part of this section to a few general properties of Wasserstein spaces (see \cite{Villani03, Villani09, santambrogio, AGS08} for a general review of Optimal Transport). The second part of the section will treat properties of Kantorovich potentials in particular geometric situations. We fix an exponent $p \in (1,+\infty)$ (recall that $p':=p/(p-1)$) and we remark that all vector spaces (and thus all Banach spaces) are \textit{real} vector spaces.\\
If $(X, \sfd)$ is a metric space, we denote by
$\prob(X)$ the space of Borel probability measures on $X$ and by
$\prob_\pd(X,\sfd)$
the set
\[ \prob_\pd(X, \sfd):= \left \{ \mu \in \prob(X) : \int_{X} \sfd^\pd(x,x_0) \d \mu(x) < + \infty \text{ for some (and hence for all) } x_0 \in X \right \}.\]
Given $\mu, \nu \in \prob(X)$ the set of transport plans between $\mu$ and $\nu$ is denoted by $\Gamma(\mu, \nu)$ and defined as
\[ \Gamma(\mu, \nu): = \left \{ \ggamma \in \prob(X \times X) : \pi^1_\sharp \ggamma=\mu, \, \pi^2_\sharp \ggamma=\nu \right \},\]
where $\pi^i(x_1,x_2)=x_i$ for every $(x_1,x_2) \in X \times X$ and $\sharp$ denotes the push forward operator as in \eqref{eq:push}. The $(p, \sfd)$-Wasserstein distance $W_{\pd,\sfd}$ between $\mu, \nu \in \prob_\pd(X, \sfd)$ is defined as
\[ W_{\pd, \sfd}^\pd(\mu, \nu) := \inf \left \{ \int_{X \times X} \sfd^\pd \,\d \ggamma : \ggamma \in \Gamma(\mu, \nu) \right \}.\]
It is well known that the infimum above is attained in a non-empty and
convex set $\Gamma_{o,p,\sfd}(\mu, \nu) \subset \Gamma(\mu, \nu)$ and that the $(p, \sfd)$-Wasserstein space $(\prob_\pd(X,\sfd), W_{\pd, \sfd})$ is complete and separable, if $(X, \sfd)$ is complete and separable. The Kantorovich duality for the Wasserstein distance states that
\begin{equation}\label{eq:duality}
    \frac{1}{p}W_{\pd, \sfd}^\pd(\mu, \nu) = \sup \left \{ \int_{X}\varphi\, \d \mu + \int_{X} \psi\, \d \nu : (\varphi, \psi) \in \text{Adm}_{\sfd^p/p}(X) \right \} \quad \text{ for every } \mu, \nu \in \prob_\pd(X, \sfd),
  \end{equation}
where $\text{Adm}_{\sfd^\pd/p}(X)$ is the set of pairs $(\varphi,\psi) \in \rmC_b(X) \times \rmC_b(X)$ such that 
\[  \varphi  (x) +  \psi  (y) \le \frac{1}{p}\sfd^\pd(x,y) \quad \text{ for every } x,y \in X.\]
It is easy to check that for every $f\in \Lip(X,\sfd)$ and $\mu, \nu \in \prob_p(X, \sfd)$ we have
\begin{equation}
  \label{eq:123}
  \int_X f\,\d(\mu-\nu)\le \Lip(f,X)W_{p, \sfd}(\mu,\nu);
\end{equation}
in fact, choosing $\ggamma\in \Gamma_{o,p, \sfd}(\mu,\nu)$ and setting $L:=\Lip(f,X)$, we have
\begin{align*}
  \int_X f\,\d(\mu-\nu)=
  \int (f(x)-f(y))\,\d\ggamma(x,y)\le
  L\int \sfd\,\d\ggamma\le
  L\Big(\int \sfd^p\,\d\ggamma\Big)^{1/p}=LW_p(\mu,\nu).
\end{align*}

\subsection{Estimates for Kantorovich
  potentials in \texorpdfstring{$(\R^d, \|\cdot\|)$}{R}}
\label{subsec:Kuseful}
In this subsection we study some stability properties for Kantorovich potentials in finite dimensional real Banach spaces for the cost induced by the $p$-th power of the norm.  We thus fix a dimension $d \in \N$ and a norm $\|\cdot\|$ on $\R^d$ such that 
\begin{equation}\label{eq:costf}
    \text{ the unit $\|\cdot\|$-ball } B_{\|\cdot\|}(0,1) := \{ x \in \R^d : \|x\|<1 \} \text{ is strictly convex and has $\rmC^{1,1}$ boundary}.
\end{equation}
We consider the map $h: \R^d \to [0, +\infty)$ defined as
\[ h(x):= \frac{1}{p}\|x\|^p, \,\, x \in \R^d, \]
and its Legendre transform $h^*:\R^d \to [0,+\infty)$ which is given by
\[ h^*(v) := \frac{1}{p'}\|v\|_*^{p'}, \quad v \in \R^d,\]
where $\|\cdot\|_*$ is the dual norm of $\|\cdot\|$. Since $h$ is strictly convex, $h^*$ is differentiable and we denote by $j_{p'}$ the gradient\footnote{Notice that in case $\|\cdot\|$ coincides with the Euclidean norm $|\cdot|$, $j_{p'}$ is simply given by 
\[ j_{p'}(v)= |v|^{p'-2}v, \quad v \in \R^d.\]} of $h^*$ which satisfies the equality
\begin{equation}\label{eq:dualityj}
\langle j_{p'}(v),v \rangle = \|v\|_*^{p'} = \|j_{p'}(v)\|^{p}, \quad \text{ for every } v \in \R^d,
\end{equation}
where $\langle \cdot, \cdot \rangle$ is the standard scalar product on $\R^d$.
The reason for restricting our analysis to norms satisfying condition \eqref{eq:costf} is that the cost function $c(x,y)$ induced by $h$ through the formula 
\begin{equation}\label{eq:inducedc}
c(x,y):= h(x-y), \quad x,y \in \R^d
\end{equation}
is compatible with the frameworks studied in \cite{Gangbo-McCann96,Gigli-Figalli} whose results we will often use. More specifically, the function $h$ as above, satisfies hypotheses (H1), (H2) and (H3) in \cite{Gangbo-McCann96}: while (H1) and (H3) are obvious, let us just mention that the smoothness of the unit $\|\cdot\|$-sphere gives that the unit $\|\cdot\|$-ball satisfies the so called $\eps$-ball condition for some $\eps>0$ (see e.g.~\cite[Definition 1.1, Theorem 1.9]{dalphin}) which in turn implies (H2). Also notice that, being the unit $\|\cdot\|$-sphere smooth, $h \in \rmC^1(\R^d)$ \cite[Proposition 13.14]{kriegl}.

We consider the complete and separable metric space $(\R^d, \sfd_{\|\cdot\|})$, where $\sfd_{\|\cdot\|}$ is the distance induced by $\|\cdot\|$, and the corresponding $(p, \sfd_{\|\cdot\|})$-Wasserstein space. To simplify the notation, in this subsection, we will simply write $\prob_p(\R^d)$, $W_p$ and $\Gamma_{o,p}$, omitting the dependence on $\sfd_{\|\cdot\|}$. Moreover we denote by $\prob_p^r(\R^d)$ the subset of $\prob_p(\R^d)$ of probability measures that are absolutely continuous w.r.t.~the $d$-dimensional Lebesgue measure $\mathcal{L}^d$ and we set
\begin{equation}
  \label{eq:183}
  \sqm\mu:=\int_{\R^d}\|x\|^p\,\d\mu(x)=W_p^p(\mu,\delta_0), \quad \mu \in \prob_p(\R^d).
\end{equation}
Finally let us set 
\[ B_{\|\cdot\|}(x,R):= \left \{ y \in \R^d : \|x-y\|< R \right \}, \quad x \in \R^d, \, R>0.\]

The next theorem uses the results of Gangbo and McCann \cite[Sections 3 and 4]{Gangbo-McCann96} and Figalli and Gigli \cite{Gigli-Figalli} to collect various useful properties of the optimal
potentials realizing the supremum in \eqref{eq:duality} when the support of one of the measures is a closed ball. This result plays the same role of \cite[Theorem 3.2]{FSS22} for the case $p=q=2$ and $\|\cdot\|$ equal to the Euclidean norm.

\begin{theorem} \label{thm:ot} Let $\mu, \nu \in \prob_p^r(\R^d)$
  with $\supp{\nu} = \overline{\rmB_{\|\cdot\|}(0,R)}$ for some $R>0$. Any pair of locally Lipschitz\footnote{Lipschitzianity is meant w.r.t.~any (equivalent) norm on $\R^d$.} functions $\phi \in L^1(\rmB_{\|\cdot\|}(0,R),\nu)$ and $\psi \in L^1(\R^d, \mu)$ such that
  \begin{enumerate}
      \item[(i)] $\displaystyle \phi(x)+ \psi(y) \le \frac{1}{p}\|x-y\|^p$ for every $(x,y) \in \rmB_{\|\cdot\|}(0,R) \times \R^d$,
      \item[(ii)] $\displaystyle \int_{\rmB(0,R)} \phi\, \d \nu + \int_{\R^d} \psi\, \d \mu =  \frac{1}{p}W_p^p(\mu, \nu)$,
  \end{enumerate}
  satisfies also
\begin{equation}
\displaystyle W_p^p(\mu, \nu) =\int_{\rmB_{\|\cdot\|}(0,R)} \left \|\nabla
  \phi \right \|_*^{p'} \d \nu=\int_{\R^d} \left \|\nabla
  \psi \right \|_*^{p'} \d \mu.\label{eq:158} 
\end{equation}
There exists a unique pair as above satisfying conditions (i), (ii) and the additional conditions
\begin{enumerate}
\item[(iii)]
  \begin{align}
    \label{eq:160}
    \displaystyle \phi(x) &= \inf_{y \in \R^d}  \left \{ \frac{1}{p}\|x-y\|^p - \psi(y)\right \}  &&\text{for every $x \in \rmB_{\|\cdot\|}(0,R)$,}\\ \label{eq:160bis}
    \displaystyle \psi(y) &= \inf_{x \in \rmB_{\|\cdot\|}(0,R)}  \left \{ \frac{1}{p}\|x-y\|^p - \phi(x)\right \}  &&\text{for every $y \in \R^d$,}
\end{align}
\item[(iv)] $\displaystyle \psi(0)=0$.
\end{enumerate}
Such a unique pair is denoted by $(\Phi(\nu, \mu), \Phi^*(\nu, \mu))$. Finally, the function $\psi:=\Phi^*(\nu, \mu)$ satisfies the following estimates: there exists a constant $K_{p,R}$, depending only on $p$ and $R$, such that
\begin{align}
\label{eq:lipestphi}
    \left |\psi(y')-\psi(y'') \right | &\le \|y'-y''\| 2^{p-1}(2R^{p-1} + \|y'\|^{p-1}+\|y''\|^{p-1}) \quad &&\text{ for every } y',y'' \in \R^d,\\
\label{eq:phiest1}
|\psi(y)| &\le K_{p,R}(1+\|y\|^{p}) \quad &&\text{ for every } y \in \R^d,\\    
\label{eq:phiest2}
\|\nabla \psi(y) \|_*  &\le K_{p,R} (1+\|y\|^{p-1}) \quad &&\text{ for $\mathcal{L}^d$-a.e. } y \in \R^d,\\ 
\label{eq:lowerboundphi}
\psi(y) &\ge \frac{1}{2p}\|y\|^p - K_{p,R} \quad &&\text{for every } y \in \R^d.
\end{align}
\end{theorem}
\begin{proof} Let $\ggamma \in \Gamma_{o,p}(\nu, \mu)$ be fixed and let $c$ be as in \eqref{eq:inducedc}.
We divide the proof in several steps.

\emph{Claim 1}. Let  $\phi \in L^1(\rmB_{\|\cdot\|}(0,R), \nu)$ and $\psi \in L^1(\R^d, \mu)$ be locally Lipschitz functions satisfying (i) and (ii). Then the maps $\tt, \ss$ defined as
\begin{align*}
    \tt(x)&:= x - (j_{p'}\circ \nabla \phi)(x) , \quad \text{$\nu$-a.e.~$x \in \rmB_{\|\cdot\|}(0,R)$},\\
    \ss(y)&:= y - (j_{p'} \circ \nabla \psi)(y), \quad \text{$\mu$-a.e.~$y \in \R^d$},
\end{align*}
are Optimal Transport maps from $\nu$ to $\mu$ and from $\mu$ to $\nu$, respectively. In particular $\phi$ and $\psi$ satisfy \eqref{eq:158}.\\
\emph{Proof of claim 1}. The proof of this statement is classical and thus omitted (see e.g.~\cite[Theorem 2.12]{Villani09}, \cite[Theorem 1.17]{santambrogio} or \cite[Theorem 6.2.4]{AGS08}).

By the proof of \cite[Theorem 6.14]{AGS08} and \cite[Theorem 6.15]{AGS08}, we have that there exists a $c$-concave function $u \in L^1(\R^d, \nu; [-\infty, + \infty))$ such that $u$ is finite on $\pi^1(\supp{\ggamma})$, $u^c$ is finite on $\pi^2(\supp{\ggamma})$, $u^c \in L^1(\R^d, \mu; [-\infty, + \infty))$ and 
\[ \int_{\R^d} u \d \nu + \int_{\R^d} u^c \d \mu = \frac{1}{p} W_p^p(\mu, \nu),\]
where $u^c$ is the $c$-transform of $u$ defined as
\[ u^c(y):= \inf_{x \in \R^d} \left \{ c(x,y)-u(x) \right \}, \quad y \in \R^d,\]
and $\pi^i: \R^d \times \R^d\to \R^d$ is the projection $\pi^i(x_1,x_2)=x_i$ for $i=1,2$.
Recall that $c$-concavity of $u$ means that there exists some proper $v:\R^d \to [-\infty, + \infty)$ such that 
\[ u(x)= \inf_{z \in \R^d} \left \{ c(x,z)-v(z) \right \}, \quad x \in \R^d.\]
\emph{Claim 2}. If we define $\tilde{\psi}$ as 
\[ \tilde{\psi}(y):= \inf_{x \in \rmB_{\|\cdot\|}(0,R)} \left \{ c(x,y)-u(x) \right \}, \quad y \in \R^d,\]
then $\psi$ is real valued and locally Lipschitz, it is finite on $\pi^2(\supp{\ggamma})$, $\tilde{\psi} \in L^1(\R^d, \mu)$ and 
\begin{equation}\label{eq:gmc1}
\int_{\R^d} u \d \nu + \int_{\R^d} \tilde{\psi} \d \mu = \frac{1}{p} W_p^p(\nu, \mu).
\end{equation}
\quad \\
\emph{Proof of claim 2}.
Since $u(x)+u^c(y)=c(x,y)$ for every $(x,y) \in \supp{\ggamma}$, then for $\mu$-a.e.~$y \in \R^d$ there exists some $x \in \rmB_{\|\cdot\|}(0,R)$ such that $u(x)+u^c(y)=c(x,y)$. This gives that $u^c = \tilde{\psi}$ $\mu$-a.e.~in $\R^d$ and proves that $\tilde{\psi} \in L^1(\R^d, \mu)$ and \eqref{eq:gmc1}. Since $\tilde{\psi} \ge u^c$ everywhere, we also have that $\tilde{\psi}$ is finite in $\pi^2(\supp{\ggamma})$. Let us prove that $\tilde{\psi}$ is real valued. Since $v$ is proper, there exists some $z_0 \in \R^d$ such that $v(z_0) \in \R$. Then 
\[ u(x) \le c(x,z_0) -v(z_0) \quad \text{ for every }  x   \in \R^d\]
so that 
\[ \tilde{\psi}(y) \ge \inf_{ x \in \rmB_{\|\cdot\|}(0,R)} \left \{ c(x,y) - c(x,z_0) \right \} + v(z_0) > - \infty\]
since the map $x \mapsto c(x,y) - c(x,z_0)$ is bounded below in $\rmB_{\|\cdot\|}(0,R)$ for every fixed $y \in \R^d$. To prove that $\tilde{\psi}$ is locally Lipschitz, it is enough to observe that, for every $M>0$ and every $y',y'' \in \rmB_{\|\cdot\|}(0,M)$, we have
\[ \left |\tilde{\psi}(y')-\tilde{\psi}(y'') \right | \le \inf_{x \in \rmB_{\|\cdot\|}(0,R)} \left | c(x,y')-c(x,y'') \right | \le \omega_c(R,M) \|y'-y''\|, \]
where $\omega_c(R,M)$ is the uniform $\|\cdot\|$-modulus of continuity of $c$ on the compact set $\overline{\rmB_{\|\cdot\|}(0,R)} \times \overline{\rmB_{\|\cdot\|}(0,M)}$.

\emph{Claim 3}. If we define $\tilde{\varphi}$ as the restriction of $u$ to $\rmB_{\|\cdot\|}(0,R)$, then we have that $\tilde{\varphi}$ is real valued and locally Lipschitz. \\
\emph{Proof of claim 3}. Let us denote the effective domain of $u$ by 
\[D(u):= \{ x \in \R^d : u(x) > - \infty \}\]
and its interior by $\Omega$. Since $u$ is $c$-concave, by \cite[Step 2 in the proof of Theorem 1]{Gigli-Figalli}, we know that for every point $ x\in D(u) \setminus \Omega$ there exists some $v_x \in \R^d\setminus\{0\}$ and $\ell_x>0$ such that the interior of the set
\[ C_x:= \left \{ y \in \R^d : \text{ there exists } t \in [0, \ell_x] \text{ s.t. } \|x+tv_x-y\| \le t/2 \right \}\]
does not intersect $D(u)$. Thus we have that $\text{int}(\overline{D(u)}) \subset \Omega$. Moreover $u$ is locally Lipschitz in $\Omega$ by \cite[Step 3 in the proof of Theorem 1]{Gigli-Figalli}. It is then enough to prove that $\rmB_{\|\cdot\|}(0,R) \subset \Omega$. Since $\pi^1(\supp{\ggamma}) \subset \supp{\nu}$ is dense in $\supp{\nu}$ and $u$ is finite on $\pi^1(\supp{\ggamma})$, then
\[ \overline{\rmB_{\|\cdot\|}(0,R)} = \supp{\nu} \subset \overline{\pi^1(\supp{\ggamma})} \subset \overline{D(u)}.\]
Finally
\[ \Omega \supset \text{int} ( \overline {D(u)}) \supset \rmB_{\|\cdot\|}(0,R),\]
so that $\rmB_{\|\cdot\|}(0,R) \subset \Omega$.

\emph{Claim 4}. If we define $\varphi: \rmB_{\|\cdot\|}(0,R) \to \R$ and $\varphi^*: \R^d \to \R$ as
\begin{align*}
    \varphi(x) &:= \tilde{\varphi}(x) + \tilde{\psi}(0), \quad x \in \rmB_{\|\cdot\|}(0,R), \\
    \varphi^*(y) &:= \tilde{\psi}(y) - \tilde{\psi}(0), \quad y \in \R^d,
\end{align*}
then they are locally Lipschitz, $\varphi \in L^1(\rmB_{\|\cdot\|}(0,R)), \nu)$, $\varphi^*\in L^1(\R^d, \mu)$ and they satisfy points (i), (ii), (iii), (iv) of the statement.\\
\emph{Proof of claim 4}. %The fact that $\varphi$ and $\varphi^*$ are real valued and locally Lipschitz follows by claim 2 and claim 3, as the fact that $\varphi \in L^1(\rmB(0,R))$ and $\varphi^* \in L^1(\R^d, \mu)$. Point (iv) is obvious, while point (ii) is a consequence of \eqref{eq:gmc1} and the fact that $\rmB(0,R)$ has full $\nu$ measure. \eqref{eq:160bis} is already contained in the definition of $\varphi^*$ and $\tilde{\psi}$.
The only nontrivial fact to be checked is \eqref{eq:160}. Let us define the modified cost function $\tilde{c}:\R^d \to [0, + \infty]$ as
\[ \tilde{c}(x,y) := \begin{cases} c(x,y) \quad &\text{ if } (x,y) \in \rmB_{\|\cdot\|}(0,R) \times \R^d, \\ + \infty \quad &\text{ else}\end{cases} \]
and the function $w: \R^d \to [-\infty, +\infty)$ as 
\[ w(z):= v(z) - \tilde{\psi}(0), \quad z \in \R^d.\]
It is thus clear that 
\begin{align*}
    w^{\tilde{c}}(x) &= \inf_{z \in \R^d} \left \{ \tilde{c}(x,z)-w(z) \right \} = \begin{cases} \varphi(x) \quad &\text{ if } x \in \rmB_{\|\cdot\|}(0,R), \\ + \infty \quad &\text{ else}, \end{cases} \\
    w^{\tilde{c} \tilde{c}}(y) &=\inf_{x \in \R^d} \left \{ \tilde{c}(x,y)-w^{\tilde{c}}(y) \right \} = \inf_{x \in \rmB_{\|\cdot\|}(0,R)} \left \{ c(x,y) - \varphi(x) \right \} = \varphi^*(y), \quad y \in \R^d,\\
    w^{\tilde{c}\tilde{c}\tilde{c}}(x) &= \inf_{y \in \R^d} \left \{ \tilde{c}(x,y) - w^{\tilde{c} \tilde{c}}(y) \right \} = \inf_{y \in \R^d} \left \{ \tilde{c}(x,y) - \varphi^*(y) \right \}, \quad x \in \R^d.
\end{align*}
By \cite[Prop. 5.8]{Villani09}, we have that $ w^{\tilde{c}}=w^{\tilde{c}\tilde{c}\tilde{c}}$ that reduces precisely to \eqref{eq:160} in case $x \in \rmB_{\|\cdot\|}(0,R)$.

\emph{Claim 5}. The pair $(\varphi, \varphi^*)$ is unique.\\
\emph{Proof of claim 5}. Suppose that $(\varphi_0, \varphi_0^*)$ is another pair of locally Lipschitz functions $\varphi_0 \in L^1(\rmB_{\|\cdot\|}(0,R), \nu)$, $\varphi^*_0\in L^1(\R^d, \mu)$ satisfying points (i), (ii), (iii) and (iv). By claim 1. the maps $\tt, \tt_0$ defined as
\begin{align*}
    \tt(x)&:= x - (j_{p'}\circ \nabla \varphi)(x), && \text{$\nu$-a.e.~$x \in \rmB_{\|\cdot\|}(0,R)$},\\
    \tt_0(x)&:= x - (j_{p'}\circ \nabla \varphi_0)(x), && \text{$\nu$-a.e.~$x \in \rmB_{\|\cdot\|}(0,R)$},
\end{align*}
are Optimal Transport maps from $\nu$ to $\mu$. Since $\varphi$ and $\varphi_0$ are (restrictions of) $c$-concave functions and both $\tt$ and $\tt_0$ push $\nu$ to $\mu$, by \cite[Theorem 4.4]{Gangbo-McCann96} we get that $\tt=\tt_0$ $\nu$-a.e.~in $\rmB_{\|\cdot\|}(0,R)$ and this gives in particular (recall that $h$ is everywhere differentiable) that $\nabla \varphi_0 = \nabla \varphi$ $\mathcal{L}^d$-a.e.~in $\rmB_{\|\cdot\|}(0,R)$. %This means that there exists some constant $k \in \R$ such that $\varphi(x)-\varphi_0(x)=k$ for every $x \in \rmB(0,R)$, so that $\varphi^*(y)-\varphi^*_0(y)=-k$ for every $y \in \R^d$. 
From this and point (iii) we get uniqueness.

\emph{Claim 6}. Let $\psi:=\Phi^*(\nu, \mu)$. Then there exists a constant $K_{p,R}$ depending only on $p$ and $R$ such that \eqref{eq:lipestphi}, \eqref{eq:phiest1}, \eqref{eq:phiest2} and \eqref{eq:lowerboundphi} hold.\\
\emph{Proof of claim 6}. \eqref{eq:lipestphi} is a consequence of the elementary inequality
\[ |a^p-b^p | \le p|a-b|(a^{p-1}+b^{p-1}), \quad a,b \ge 0,\]
and of the fact that 
\[ |\psi(y') - \psi(y'')| \le \frac{1}{p} \sup_{x \in \rmB_{\|\cdot\|}(0,R)} \left | \|x-y'\|^p - \|x-y''\|^p \right |.\]
\eqref{eq:phiest1}, \eqref{eq:phiest2} directly follow by \eqref{eq:lipestphi}. %, respectively choosing $y''=0$, and dividing by $|y'-y''|$ and passing to the limit as $y' \to y''$ for some $y''$ where $\psi$ is differentiable. 
Finally, \eqref{eq:lowerboundphi} follows by the inequality
\begin{equation}\label{eq:ineqtt}
\|x-y\|^p \ge \frac{1}{2} \|y\|^p - K_p\|x\|^p, \quad x,y\in\R^d,
\end{equation}
that holds for a suitable constant $K_p>0$ that depends solely on $p$. In fact, using (iii), we have that
\[ \varphi(x) \le \frac{1}{p}\|x\|^p \quad \text{ for every } x \in \R^d\]
so that 
\[ \psi(y) \ge \inf_{x \in \rmB_{\|\cdot\|}(0,R)} \left \{ \frac{1}{p}\|x-y\|^p - \frac{1}{p} \|x\|^p \right \} \ge \frac{1}{2p}\|y\|^p - \frac{1}{p}(K_p+1)R^p \quad \text{ for every } y \in \R^d.\]
\end{proof}
\begin{remark} There exists a constant $D_{p,R}$, depending only on $p$ and $R$ such that, if $\psi:\R^d \to \R$ is a function satisfying \eqref{eq:lipestphi}, then
\begin{equation}\label{eq:w2lipphi}
\int_{ \R^d } \psi \d (\mu'-\mu'') \le D_{p,R} W_p(\mu', \mu'')(1+\rsqm{\mu'}+\rsqm{\mu''}) \quad \text{ for every } \mu', \mu'' \in \prob_p(\R^d).    
\end{equation}
In fact, by \eqref{eq:lipestphi}, if we take any $\ggamma \in \Gamma_{o,p}(\mu', \mu'')$, we have
\begin{align*}
    \int_{\R^d} \psi \d (\mu'-\mu'') &= \int_{\R^d \times \R^d} \left (\psi(y')-\psi(y'') \right ) \d \ggamma (y',y'') \\
    & \le 2^{p-1}\left ( \int_{R^d \times \R^d} \|y'-y''\|^p \d \ggamma(y',y'') \right )^{1/p} \\
    &\quad \left ( \int_{\R^d \times \R^d} \left (2R^{p-1} + \|y'\|^{p-1}+\|y''\|^{p-1} \right )^{p'} \d \ggamma(y',y'') \right )^{1/p'}
\end{align*}
so that \eqref{eq:w2lipphi} follows.
\end{remark}

In the following results we collect some properties of sequences of pairs of potentials as in Theorem \ref{thm:ot}. The aim is to prove that, under suitable conditions, such sequences converge to optimal potentials.

In the next Propositions we use the notation \[\dist(A,B) := \inf \{ \|x-y\| : x \in A, y \in B \}, \quad A,B \subset \R^d.\]

\begin{proposition}\label{prop:oldcl23} Let $R,I>0$ and let $(\varphi_n)_n, (\psi_n)_n$ be sequences of functions such that
\begin{align*}
\varphi_n(x)&= \inf_{y \in \R^d} \left \{ \frac{1}{p}\|x-y\|^p - \psi_n(y) \right \}, \quad &&\psi_n(y) = \inf_{x \in \rmB_{\|\cdot\|}(0,R)} \left \{ \frac{1}{p}\|x-y\|^p - \varphi_n(y) \right \}, \\
\psi_n(0)&=0, \quad &&\int_{\rmB_{\|\cdot\|}(0,R)} \varphi_n \, \d \mathcal{L}^d \ge - I
\end{align*}
for every $x \in \rmB_{\|\cdot\|}(0,R)$, every $y \in \R^d$ and every $n \in \N$. Then $\varphi_n$ is locally (w.r.t.~$x \in \rmB_{\|\cdot\|}(0,R)$) uniformly (w.r.t.~$n \in \N$) bounded and Lipschitz.
\end{proposition}
\begin{proof} The proof is strongly based on \cite{Gigli-Figalli} (see also the similar approach in \cite[Proposition C.3]{Gangbo-McCann96}) and it is divided in two claims. 
For the whole proof, we extend $\varphi_n$ to $\R^d$ simply setting
\[ \varphi_n(x) = \inf_{y \in \R^d}  \left \{ \frac{1}{p}\|x-y\|^p-\psi_n(y) \right\} \ \text{for every } x \in \R^d.\]
\quad \\
\emph{Claim 1}. The sequence $\varphi_n$ is locally (w.r.t.~$x \in \rmB_{\|\cdot\|}(0,R)$) uniformly (w.r.t.~$n \in \N$) bounded.\\
\emph{Proof of claim 1}. Since $\varphi_n(x) \le \frac{1}{p}\|x\|^p$ for every $x \in \R^d$, it is enough to prove that $\varphi_n$ is locally uniformly bounded from below in $\rmB_{\|\cdot\|}(0,R)$. Let us suppose by contradiction that there exist some $\bar{x} \in \rmB_{\|\cdot\|}(0,R)$ and a sequence $(x_n)_n \subset \rmB_{\|\cdot\|}(0,R)$ such that $x_n \to \bar{x}$ and $\varphi_n(x_n) \to - \infty$. For every $n \in \N$ we can  choose  some $y_n \in \R^d$ such that
\[ \varphi_n(x_n) + 1 \ge \frac{1}{p}\|x_n-y_n\|^p - \psi_n(y_n) \ge - \psi_n(y_n)\]
implying that $\psi_n(y_n) \to + \infty$. Thus, since $\varphi_n(\bar{x}) \in \R$ and 
\[ \varphi_n(\bar{x}) \le \frac{1}{p}\|\bar{x}-y_n\|^p - \psi(y_n),\]
we get that $\|\bar{x}-y_n\| \to + \infty$ which in particular gives that $\|x_n-y_n\|\to + \infty$. Let us define now the curves $\gamma_n: [0,\|x_n-y_n\|]\to \R^d$ as
\[ \gamma_n(t)= \left (1- \frac{t}{\|x_n -y_n\|} \right ) x_n + \frac{t}{\|x_n -y_n\|} y_n, \quad t \in [0,\|x_n-y_n\|].\]
Since $\|x_n-y_n\| \to + \infty$, we can assume that $\|x_n-y_n\| \ge 1$ for every $n \in \N$ and define
\[ C_n := \left \{ x \in \R^d : \text{ there exists } t \in [0,1] \text{ such that } \|x-\gamma_n(t)\| \le \frac{t}{2} \right \}, \quad n \in \N.\]
We claim that $\sup_{C_n} \varphi_n \to -\infty$ as $n \to + \infty$. Indeed if $x \in \R^d$ and $t \in [0,1]$ is such that $\|x-\gamma_n(t)\| \le t/2$, we have
\begin{align}\begin{split}\label{eq:splitcos}
    \varphi_n(x) &\le \frac{1}{p}\|x-y\|^p - \psi_n(y_n) \\
    & \le \frac{1}{p} \left ( \|x-\gamma_n(t)\| + \|\gamma_n(t)-y_n\| \right )^p - \psi_n(y_n)\\
    & \le \frac{1}{p} \left ( t/2 + \|\gamma_n(t)-y_n\| \right )^p - \psi_n(y_n)\\
    &= \frac{1}{p} \left | \|x_n-y_n\| - t/2  \right |^p - \psi_n(y_n) \\
    & \le \frac{1}{p}\|x_n-y_n\|^p - \frac{t}{2}\|x_n-y_n\|^{p-1} - \psi_n(y_n) \\
    & \le \varphi_n(x_n) + \frac{1}{n} - \frac{t}{2} \|x_n-y_n\|^{p-1} \\
    & \le \varphi_n(x_n)+1,\end{split}
\end{align}
where we have used the elementary inequality
\[ (a-b)^p \le a^p -pa^{p-1}b^p \quad \text{ for every } 0 \le a \le b.\]
Letting $n \to + \infty$, we obtain the claim. Up to passing to a (unrelabeled) subsequence, we can thus assume that
\[ \varphi_n(x) \le -n \quad \text{for every } x \in C_n, \quad n \in \N.\]
Since $\bar{x} \in \rmB_{\|\cdot\|}(0,R)$ there exists some $\eps>0$ such that $\rmB_{\|\cdot\|}(\bar{x}, \eps) \subset \rmB_{\|\cdot\|}(0,R)$. Let $N \in \N$ be such that $\|x_n-\bar{x}\|< \eps/2$ for every $n \ge N$ and let $\delta:= \frac{\eps}{3}$; it is not difficult to check that the truncated sets
\[ C_n^{\delta} := \left \{ x \in \R^d : \text{ there exists } t \in [0,\delta] \text{ such that } \|x-\gamma_n(t)\| \le \frac{t}{2} \right \} \subset C_n, \quad n \in \N\]
are such that $C_n^\delta \subset \rmB_{\|\cdot\|}(\bar{x}, \eps) \subset \rmB_{\|\cdot\|}(0,R)$ for every $n \ge N$. Then for every $n \ge N$ we have
\begin{align*}
    -I-\frac{1}{p}\sqm{\mathcal{L}^d\mres\rmB_{\|\cdot\|}(0,R)} &\le \int_{\rmB_{\|\cdot\|}(0,R)} \left (\varphi_n(x) - \frac{1}{p}\|x\|^p \right ) \d \mathcal{L}^d(x)\\
    &\le \int_{C_n^\delta} \left (\varphi_n(x)-\frac{1}{p}\|x\|^p\right ) \d \mathcal{L}^d(x)\\
    & \le \int_{C_n^\delta} \left (-n-\frac{1}{p}\|x\|^p \right ) \d \mathcal{L}^d(x)\\
    & \le - n \mathcal{L}^d(C_n^\delta) \\
    & = - \frac{n \omega_{d-1}\delta^d}{d2^{d-1}\eta^{d-1}},
\end{align*}
where $\omega_{d-1}$ is the $\mathcal{L}^{d-1}$-measure of the Euclidean $(d-1)$-dimensional unit ball and $\eta>0$ is a constant such that
\[ \|x\| \le \eta |x| \quad \text{ for every } x \in \R^d,\]
where $|\cdot|$ is the Euclidean norm on $\R^d$.
Passing to the limit as $n \to + \infty$ gives that $I=\infty$, a contradiction.

\emph{Claim 2}. For every $K \subset \subset \rmB_{\|\cdot\|}(0,R)$ there exists $M_K>0$ such that 
\[ \varphi_n(x) = \inf_{y \in \rmB_{\|\cdot\|}(0, M_K)} \left \{ \frac{1}{p}\|x-y\|^p - \psi_n(y)\right \} \quad \text{for every $x \in K$ and for every $n \in \N$, }\]
so that the sequence $\varphi_n$ is locally (w.r.t.~$x \in \rmB_{\|\cdot\|}(0,R)$) uniformly (w.r.t.~$n \in \N$) Lipschitz.\\
\emph{Proof of claim 2}. Let $K \subset \subset \rmB_{\|\cdot\|}(0,R)$ be fixed; we claim that there exists $C_K>0$ such that, if $(x,y,n) \in K \times \R^d \times \N$ are such that
\[ \varphi_n(x) +1 \ge \frac{1}{p}\|x-y\|^p - \psi_n(y),
\]
then $y \in \overline{\rmB_{\|\cdot\|}(x,C_K)}$. Let us define
\[ C_K := \max \left \{ 1, \left (\frac{2}{\ell} \left (S_{K_\ell}-I_{K_\ell} +1 \right )\right)^{\frac{1}{p-1}} \right  \}, \]
where $0 <\ell < \dist(K, \rmB^c_{\|\cdot\|}(0,R))$ and
\begin{align*}
K_\ell &:= \left \{ x \in \R^d : \dist(x,K) \le \ell \right \} \subset \rmB_{\|\cdot\|}(0,R),\\
S_{K_\ell} &:= \sup \left \{ \varphi_n(x) : (x,n) \in K_\ell \times \N \right \} < + \infty,\\
I_{K_\ell} &:= \inf \left \{ \varphi_n(x) : (x,n) \in K_\ell \times \N \right \} > -\infty,
\end{align*}
where we have used claim 1 to ensure that the supremum and the infimum are finite. Let us consider $(x,y,n) \in K \times \R^d \times \N$ as above. To prove that $C_K$ is the sought constant, it is enough to consider the case in which $\|x-y\|\ge 1$. If we define $\gamma :[0, \|x-y\|]\to \R^d$ as the curve given by
\[ \gamma(t):= \left ( 1- \frac{t}{\|x-y\|} \right )x +  \frac{t}{\|x-y\|} y, \quad t \in [0, \|x-y\|],\]
we have, arguing as in \eqref{eq:splitcos}, that
\[ \varphi_n(\gamma(\ell)) \le \varphi_n(x) + 1 - \frac{\ell}{2} \|x-y\|^{p-1}\]
so that $\|x-y\| \le C_K$. We thus set 
\[ M_K := \sup \left \{\|y\| : y \in K+ \rmB_{\|\cdot\|}(0, C_K) \right \}.\]
If $x \in K$ and $n \in \N$ are fixed, we can find a minimizing sequence $(y_k)_k \subset \R^d$, in the sense that 
\[ \varphi_n(x) + 1 \ge \varphi_n(x_k) + \frac{1}{k} \ge \frac{1}{p} \|x-y_k\|^p-\psi_n(y_k) \quad \text{ for every } k \in \N\]
so that, by the result we have just proven, $(y_k)_k$ is bounded and thus converges, up to a (unrelabeled) subsequence to some $\bar{y} \in \overline{\rmB_{\|\cdot\|}(x, C_K)}$. Passing the above inequality to the limit as $k \to + \infty$, we get that 
\[ \varphi_n(x) \ge \frac{1}{p}\|x-\bar{y}\|^p - \psi_n(\bar{y}).\]
This proves that $\bar{y}$ is a minimizer and, since it belongs to $\rmB_{\|\cdot\|}(0, M_K)$, we get the second claim.
\end{proof}

\begin{proposition}\label{prop:oldcl4567} Let $R,I>0$ and let $(\varphi_n)_n, (\psi_n)_n$ be sequences of functions such that
\begin{align*}
\varphi_n(x)&= \inf_{y \in \R^d} \left \{ \frac{1}{p}\|x-y\|^p - \psi_n(y) \right \}, \quad &&\psi_n(y) = \inf_{x \in \rmB_{\|\cdot\|}(0,R)} \left \{ \frac{1}{p}\|x-y\|^p - \varphi_n(y) \right \}, \\
\psi_n(0)&=0, \quad &&\int_{\rmB_{\|\cdot\|}(0,R)} \varphi_n \, \d \mathcal{L}^d \ge - I
\end{align*}
for every $x \in \rmB_{\|\cdot\|}(0,R)$, every $y \in \R^d$ and every $n \in \N$. Then there exist a subsequence $j \mapsto n(j)$ and two locally Lipschitz functions $\varphi:\rmB_{\|\cdot\|}(0,R) \to \R$ and $\psi: \R^d \to \R$ such that $\varphi_{n(j)} \to \varphi$ locally uniformly in $\rmB_{\|\cdot\|}(0,R)$, $\psi_{n(j)} \to \psi$ locally uniformly in $\R^d$ and 
\begin{align}
\partial^c\psi(y) \ne \emptyset, \quad  \lim_j \sup \left \{ \dist(x, \partial^c \psi(y)) : x \in \partial^c \psi_{n(j)}(y) \right \} =0 \quad \text{ for every } y \in \R^d, \\ \label{eq:convgrad}
\nabla \psi_{n(j)} \to \nabla \psi \quad \text{ $\mathcal{L}^d$-a.e.~in $\R^d$},
\end{align}
where  $\partial^c \psi$ is the $c$-superdifferential operator of $\psi$, whose graph is given by
\begin{equation}\label{eq:csuperd}
 \partial^c \psi := \left \{ (y,x) \in \R^d \times \R^d : \psi(z) \le \psi(y) + \frac{1}{p}\|z-x\|^p - \frac{1}{p} \|y-x\|^p \text{ for every } z \in \R^d \right \}.
\end{equation}
\end{proposition}
\begin{proof} The proof is divided in three claims, the first of which is adapted from \cite[Proposition C.4]{Gangbo-McCann96}.
By Proposition \ref{prop:oldcl23} we have that the sequence $\varphi_n$ is locally (w.r.t.~$x \in \rmB_{\|\cdot\|}(0,R)$) uniformly (w.r.t.~$n \in \N$) bounded and Lipschitz. Arguing as in Theorem \ref{thm:ot} (see in particular \eqref{eq:lipestphi} and
\eqref{eq:phiest1}), we have that the sequence $\psi_n$ is locally (w.r.t.~$y \in \R^d$) uniformly (w.r.t.~$n \in \N$) bounded and Lipschitz. Hence we can apply Ascoli-Arzelà theorem and obtain the existence of a subsequence $j \mapsto n(j)$ and two locally Lipschitz functions $\varphi:\rmB_{\|\cdot\|}(0,R)$ and $\psi: \R^d \to \R$ such that $\varphi_{n(j)} \to \varphi$ locally uniformly in $\rmB_{\|\cdot\|}(0,R)$ and $\psi_{n(j)} \to \psi$ locally uniformly in $\R^d$.

\emph{Claim 1}. For every $K>0$ there exists a constant $C_K>0$ such that
\[ \partial^c \psi_n( \overline{\rmB_{\|\cdot\|}(0,K)}) \subset \overline{\rmB_{\|\cdot\|}(0,C_K)} \times \overline{\rmB_{\|\cdot\|}(0,C_K)} \quad \text{ for every } n \in \N, \]
where $\partial^c \psi_n$ is the $c$-superdifferential operator of $\psi_n$ (which is defined analogously to $\partial^c \psi$ adapting \eqref{eq:csuperd} in the obvious way).
Moreover $\partial^c\psi_n(y) \ne \emptyset$ for every $y \in \R^d$ and every $n \in \N$.\\
\emph{Proof of claim 1}. Let us start from the last part of the claim; let $y \in \R^d$ and $n \in \N$ be fixed. By hypothesis, we can find a sequence $(x_j)_j \subset \rmB_{\|\cdot\|}(0,R)$ such that for every $j \in \N$ it holds
\begin{align*}
    \frac{1}{p}\|x_j-y\|^p - \varphi_n(x_j) -\frac{1}{j} &\le \psi_n(y), \\
    \psi_n(z) &\le \frac{1}{p}\|x_j-z\|^p - \varphi_n(x_j) \quad \text{ for every } z \in \R^d.
\end{align*}
This gives that 
\[ \psi_n(z) + \frac{1}{p}\|x_j-y\|^p - \frac{1}{j} \le \psi_n(y) + \frac{1}{p}\|x_j-z\|^p \quad \text{ for every } (z,j) \in \R^d \times \N.\]
Up to a (unrelabeled) subsequence, there exists some $x \in \R^d$ such that $x_j \to x$ and we get passing to the limit the above inequality that
\[ \psi_n(z) \le \psi_n(y) + \frac{1}{p}\|x-z\|^p - \frac{1}{p}\|x-y\|^p \quad \text{ for every } z \in \R^d,\]
meaning that $x \in \partial^c\psi_n(y)$. Let us come to the first part of the claim. Let $K>0$ be fixed and observe that, arguing as in Theorem \ref{thm:ot} (see in particular \eqref{eq:phiest1}), the sequence $\psi_n$ is locally (w.r.t.~$y \in \R^d$) uniformly (w.r.t.~$n \in \N$) bounded, so that there exists some $T_K>0$ such that
\begin{equation}\label{eq:uniboundt}
\sup_{n \in \N} \sup_{y \in \overline{\rmB_{\|\cdot\|}(0,K+\frac{1}{2})}} |\psi_n(y) | < T_K.
\end{equation}
We claim that there exists a constant $D_K$ such that, whenever $(x,\lambda,u,n) \in \R^d \times \R \times \overline{\rmB_{\|\cdot\|}(0,K)}\times \N$ are such that
\begin{align}\label{eq:gm1}
    \psi_n(z) &\le \frac{1}{p}\|z-x\|^p + \lambda \quad \text{ for every } z \in \R^d,\\
    \frac{1}{p}\|u-x\|^p + \lambda &<T_K, \label{eq:gm2}
\end{align}
then $\|x\|\le D_K$. Suppose by contradiction that we can find a sequence $(x_k,\lambda_k,u_k,n_k)_k \subset \R^d \times \R \times \overline{\rmB_{\|\cdot\|}(0,K)}\times \N$ satisfying \eqref{eq:gm1} and \eqref{eq:gm2} such that $\|x_k\| \to + \infty$. Let $v_k:=u_k-x_k$ for every $k \in \N$. Since $u_k \in \overline{\rmB_{\|\cdot\|}(0,K)}$, and $\|x_k\| \to + \infty$, we may assume that $\|v_k\|>1$ for every $k \in \N$. Let us also define $\xi_k:= 1- \frac{1}{2\|v_k\|}$ so that $\xi_k \to 1$ as $k \to + \infty$. Finally, let $z_k:= u_k + (\xi_k-1)v_k$, $k \in \N$ and observe that \eqref{eq:uniboundt} implies that $|\psi_{n_k}(z_k)| < T_K$ for every $k \in \N$ since
\[ \|z_k\| = \left \|u_k + (\xi_k-1)v_k \right \| = \left \|u_k + \frac{v_k}{2\|v_k\|} \right \| \le K + \frac{1}{2}.\]
We can thus evaluate \eqref{eq:gm1} and \eqref{eq:gm2} written for $(x_k,\lambda_k,u_k,n_k)$ with $z=z_k$ obtaining that
\[
    \psi_{n_k}(z_k) \le \frac{1}{p}\|z_k-x_k\|^p + \lambda_k, \quad \frac{1}{p}\|u_k-x_k\|^p + \lambda_k <T_K \quad \text{ for every } k \in \N.
\]
We thus deduce that 
\begin{equation}\label{eq:usefordiff}
\frac{1}{p} \|v_k\|^p - \frac{1}{p} \|\xi_k v_k\|^p = \frac{1}{p}\|v_k\|^p - \frac{1}{p}\|z_k-x_k\|^p \le 2T_K \quad \text{ for every } k \in \N.
\end{equation}
Let now $w_k$ be the unique element of the subdifferential of $h$ at the point $\xi_k v_k$ so that 
\begin{align}\label{eq:laprimdiff}
\frac{1}{p}\|0\|^p - \frac{1}{p}\|\xi_kv_k\|^p \ge \la w_k, 0-\xi_kv_k \ra \quad \text{ for every } k \in \N,\\
\frac{1}{p}\|v_k\|^p - \frac{1}{p}\|\xi_kv_k\|^p \ge \la w_k, v_k-\xi_kv_k \ra \quad \text{ for every } k \in \N. \label{eq:lasecdiff}
\end{align}
Combining \eqref{eq:usefordiff} with \eqref{eq:lasecdiff} we get
\begin{align*}
    2T_K &\ge \frac{1}{p}\|v_k\|^p - \frac{1}{p}\|\xi_kv_k\|^p \\
    &\ge \la w_k, v_k-\xi_kv_k \ra \\
    &= \frac{1}{2\|v_k\|}\la w_k , v_k \ra  \\
    &\ge \frac{1}{2p}\xi_k^{p-1} \|v_k\|^{p-1} \quad \text{ for every } k \in \N,
\end{align*}
where for the last inequality we have used \eqref{eq:laprimdiff}. Since $\xi_k \ge \frac{1}{2}$ and $\|v_k\| \to + \infty$, we get a contradiction, proving our first claim. We set $C_K:= \max\{K, D_K\}$ and we prove that $\partial^c \psi_n( \overline{\rmB_{\|\cdot\|}(0,K)}) \subset \overline{\rmB_{\|\cdot\|}(0,C_K)} \times \overline{\rmB_{\|\cdot\|}(0,C_K)}$ for every $n \in \N$. If $(y,x) \in \partial^c \psi_n(\overline{\rmB_{\|\cdot\|}(0,K)}$ and we define $\lambda:=\psi_n(y)-\frac{1}{p}\|x-y\|^p$, we have that $(x,\lambda, y, n) \in \R^d \times \R \times \overline{\rmB_{\|\cdot\|}(0,K)}\times \N$ and satisfies \eqref{eq:gm1} and \eqref{eq:gm2} by the very definition of $\partial^c \psi_n$ so that by the above claim we get that $\|x\|\le D_K \le C_K$ and of course $\|y\| \le K \le C_K$. This concludes the proof of the claim.

\emph{Claim 2}. We have that 
\[ \partial^c\psi(y) \ne \emptyset, \quad  \lim_j \sup \left \{ \dist(x, \partial^c \psi(y)) : x \in \partial^c \psi_{n(j)}(y) \right \} =0 \quad \text{ for every } y \in \R^d.\]
\emph{Proof of claim 2}. Let $y \in \R^d$ be fixed; observe that, if $(x_j)_j$ is any sequence such that $x_j \in \partial^c \psi_{n(j)}(y)$ for every $j \in \N$ (there exists at least one such a sequence by claim 1), then by claim 1 $\partial^c\psi_{n(j)}(y)$ are uniformly (w.r.t~$j \in \N$) bounded so that we can extract a subsequence $k\mapsto j_k$ such that $x_{j_k} \to x$ for some $x \in \R^d$. By the very definition of $c$-superdifferential, we have that for every $k \in \N$ it holds
\[ \psi_{n(j_k)}(z) \le \psi_{n(j_k)}(y) + \frac{1}{p}\|z-x_{j_k}\|^p - \frac{1}{p}\|y-x_{j_k}\|^p \quad \text{ for every } z \in \R^d.\]
Passing to the limit as $k \to + \infty$, we get that $x \in \partial^c \psi(y)$, so that $\partial^c \psi(y)\ne \emptyset$. This proves in particular that from any sequence of points $(x_j)_j$ such that $x_j \in \partial^c \psi_{n(j)}(y)$ for every $j \in \N$ we can extract a subsequence converging to an element of $\partial^c \psi(y)$. Let us come now to the proof of the limit. We prove that from any (unrelabeled) subsequence of $n(j)$ we can extract a further subsequence such that we have the convergence above. For every $j \in \N$ we can find some $x_j \in \partial^c \psi_{n(j)}(y)$ such that 
\[ \dist(x_j, \partial^c\psi(y)) + \frac{1}{j} \ge \sup \left \{ \dist(x, \partial^c \psi(y)) : x \in \partial^c \psi_{n(j)}(y) \right \}.\]
Reasoning as above, we find a subsequence $k \mapsto j_k$ such that $x_{j_k} \to x \in \partial^c\psi(y)$ as $k \to + \infty$. Then we have
\begin{align*}
0=\dist(x,\partial^c\psi(y))&= \lim_k \dist(x_{j_k}, \partial^c\psi(y)) + \frac{1}{j_k} \\
&\ge \limsup_k \sup \left \{ \dist(x, \partial^c \psi(y)) : x \in \partial^c \psi_{n(j_k)}(y) \right \}. 
\end{align*}
This concludes the proof of the claim.

\emph{Claim 3}. We have that $\nabla \psi_{n(j)} \to \nabla \psi$ $\mathcal{L}^d$-a.e.~in $\R^d$.\\
\emph{Proof of claim 3}. Let $A:= \left \{ y \in \R^d : \text{$\psi$ and $\psi_{n(j)}$ are differentiable at $y$ for every $j \in \N$} \right \}$; notice that $A$ has full $\mathcal{L}^d$ measure. Let $y \in A$; since $\psi$ is differentiable at $y$ we know by \cite[Lemma 3.1]{Gangbo-McCann96} that any $x \in \partial^c \psi(y)$ will satisfy $\|\nabla \psi(y)\|_*^{p'} =\|\nabla h(x-y)\|_*^{p'}=\|x-y\|^p$ so that $\partial^c \psi(y)$ is contained in $\overline{\rmB_{\|\cdot\|}(0, C)}$ for some $C>0$. On the other hand, claim 1 implies that there exists some constant $D>0$ such that $\partial^c \psi_{n(j)}(y) \subset \overline{\rmB_{\|\cdot\|}(0, D)}$ for every $j \in \N$. Let us define $M>0$ as the uniform modulus of continuity of the map $x \mapsto \nabla h(x-y)$ in the compact set $\overline{\rmB_{\|\cdot\|}(0, C+D)}$. By claim 2, for every $\eps>0$ we can find $J_\eps \in \N$ such that, if $j \ge J_\eps$, then 
\[ \dist(x, \partial^c \psi(y)) \le \sup \left \{ \dist(x, \partial^c \psi(y)) : x \in \partial^c \psi_{n(j)}(y) \right \} < \eps/M \quad \text{ for every } x \in \partial^c \psi_{n(j)}(y).\]
This means that for every $j \ge J_\eps$ and every $x_j \in \partial^c \psi_{n(j)}(y)$ there exists some $z_j \in \partial^c \psi(y)$ such that $\|x_j-z_j\| < \eps/M$. Using again \cite[Lemma 3.1]{Gangbo-McCann96} and the fact that both $\psi_{n(j)}$ and $\psi$ are differentiable at $y$, we have that $\nabla \psi_{n(j)}(y)= \nabla h(x_j-y) $ and $\nabla \psi(y)= \nabla h(z_j-y)$ for every $j \ge J_\eps$. Then
\[ \left | \nabla \psi_{n(j)}(y) - \nabla \psi(y) \right | = \left |\nabla h(x_j-y) - \nabla h(z_j-y) \right | \le M \|x_j-z_j\| < \eps\]
for every $j \ge J_\eps$. This concludes the proof of the claim.
\end{proof}

\section{The Wasserstein Sobolev space
  \texorpdfstring{$H^{1,q}(\prob_p(X, \sfd),W_{p, \sfd},\mm)$}{H}}
\label{sec:main2}
The aim of this section is to study the $q$-Sobolev space on the $(p, \sfd)$-Wasserstein space on a separable and complete metric space $(X,\sfd)$. In particular we will show at the end of this section that the algebra of cylinder functions generated by a sufficiently rich algebra of Lipschitz and bounded functions on $(X,\sfd)$ is dense in $q$-energy in the Sobolev space $H^{1,q}(\prob_p(X, \sfd),W_{p, \sfd},\mm)$. For the whole section $p,q \in (1,+\infty)$ are fixed exponents.

\subsection{Cylinder functions on \texorpdfstring{$(X, \sfd)$}{B} and their differential in Banach spaces}
\label{subsec:cylindrcial}
Let $(X,\sfd)$ be a complete and separable metric space. We extend the definition of cylinder function in \cite{FSS22} to this more general setting. To every $\phi \in \Lipb(X,\sfd)$ we can associate the functional $\lin\phi$ on $\prob(X)$
\begin{equation}\label{eq:monocyl}
  \lin\phi:\mu \to \int_{X} \phi \,\d \mu
\end{equation}
which belongs to $\Lip_b(\prob_p(X, \sfd),W_{p, \sfd})$ thanks to \eqref{eq:123}. If $\pphi=(\phi_1,\cdots,\phi_N)\in
\big(\Lipb(X,\sfd)\big)^N$, we denote by 
$\lin\pphi:=(\lin{\phi_1},\cdots,\lin{\phi_N})$ the corresponding map
from $\prob(X)$ to $\R^N$.

\begin{definition}[$\EE$-cylinder functions]\label{def:cyl} Let $\EE \subset \Lipb(X, \sfd)$ be an algebra of functions; we say that a function $F: \prob(X) \to \R$ is a $\EE$-\emph{cylinder function} if there exist $N \in \N$, $\psi \in \rmC_b^1(\R^N)$ and $\uphi= (\phi_1, \dots, \phi_N) \in \EE^N$ such that 
\begin{equation}\label{eq:cyl}
F(\mu) = \psi(\lin\pphi(\mu))=\psi \big( \lin{\phi_1}(\mu),\cdots,\lin{\phi_N}(\mu)\big)
\quad \text{for every } \mu \in \prob(X).
\end{equation}
We denote the set of such functions by $\ccyl{\prob(X)}{\EE}$.
\end{definition}

\begin{remark}
  Since for every $\uphi\in \EE^N$
  the range of $\lin\uphi$ is always contained in the
  bounded set
  $[-M,
  M]^N$ where $M:= \max_{i=1, \dots, d} \|\phi_i\|_\infty$,
  also functions $F=\psi\circ \lin\uphi$ % : \prob_2(\R^d) \to \R$ of the form
% \[ F(\mu) = \psi \left ( \la \uphi, \mu \ra \right ) \quad \forall \mu \in \prbt\]
  with $\psi \in \rmC^1(\R^N)$ belong
  to $\ccyl{\prob(X)}{\EE}$. Indeed it is enough to consider a function
  $\tilde{\psi} \in \rmC_b^1(\R^N)$ coinciding with $\psi$ on
   $[-M,
  M]^N$ and equal to $0$ outside $[-M-1, M+1]^N$
  so that $F=\tilde\psi\circ \lin\uphi$.
% \[F(\mu) = \tilde{\psi} \left ( \la \uphi, \mu \ra \right ) \quad \forall \mu \in \prbt.\]
In particular
every function of the form $\lin\phi$, $\phi\in \EE$,
belongs to $\ccyl{\prob(X)}{\EE}$.

\end{remark}
We use the notation $\domG_X$ for the Borel set
\begin{equation}
  \label{eq:51}
  \domG_X:=\Big\{(\mu,x)\in \prob(X) \times X:x\in \supp(\mu)\Big\}.
\end{equation}
We introduce now the definition of differential of a cylinder function, still following \cite{FSS22}, in case $(X,\sfd)=(\B, \sfd_{\|\cdot\|})$, where $(\B, \|\cdot\|)$ is a Banach space and $\sfd_{\|\cdot\|}$ is the distance induced by $\|\cdot\|$. In this case we denote by $\rmC^1_b(\B)$ the space of bounded, continuously differentiable and Lipschitz functions on $\B$.
\begin{definition} Let $F\in \ccyl{\prob(\B)}{\rmC_b^1(\B)}$; then, given $N \in \N$, $\psi \in \rmC_b^1(\R^N)$ and $\uphi \in (\rmC_b^1(\B))^N$ such that $F=\psi\circ\lin\uphi$, we define the Wasserstein differential of $F$ conditioned to $(\psi, \uphi)$, 
$\rmD F_{\psi, \pphi}: \overline{\domG_\B}\to \B^*$, as
\begin{align}\label{eq:Fdiff}
\rmD F_{\psi, \pphi} (\mu,x) &:= \sum_{n=1}^N \partial_n \psi \left ( \lin \uphi(\mu)
                 \right ) \nabla \phi_n(x), \quad (\mu,x)\in \overline{\domG_\B}.
                \intertext{We will also denote by $\rmD F_{\psi, \pphi}[\mu]$ the
                function $x\mapsto \rmD F_{\psi, \pphi}(\mu,x)$ and we will set}
\| \rmD F_{\psi, \pphi}[\mu] \|_{*,p',\mu} &:= \left (\int_{\B} \|\rmD F_{\psi, \pphi}[\mu](x)\|_*^{p'} \d \mu(x) \right )^{1/p'}, \quad \mu \in \prob_p(\B, \sfd_{\|\cdot\|}),
\end{align}
where $\|\cdot\|_*$ is the dual norm to $\|\cdot\|$.
\end{definition}

\begin{remark} Let $\mu_0, \mu_1 \in \prob_p(\B, \sfd_{\|\cdot\|})$, let $\mmu \in \Gamma_{o,p, \sfd_{\|\cdot\|}}(\mu_0, \mu_1)$ and let $u \in L^p(\B, \mu_0; (\B, \|\cdot\|))$. We define the curves $\mu, \nu: [0,1] \to \prob_p(\B, \sfd_{\|\cdot\|})$ as
\begin{align}\label{eq:curves}
    \mu_t&:= \sfx^t_\sharp \mmu, \quad &&t \in [0,1], \\\label{eq:curves1}
    \nu_t&:= (\ii_\B+tu)_\sharp \mu_0, \quad &&t \in [0,1],
\end{align}
 where $\sfx^t: \B \times \B \to \B$ is the map defined as $\sfx^t(x_0,x_1):= (1-t)x_0+tx_1$, for every $(x_0, x_1) \in \B \times \B$ and $t \in [0,1]$. If $F\in \ccyl{\prob(\B)}{\rmC_b^1(\B)}$, given $N \in \N$, $\psi \in \rmC_b^1(\R^N)$ and $\uphi \in (\rmC_b^1(\B))^N$ such that $F=\psi\circ\lin\uphi$, then 
 \begin{align}
 \label{eq:derivative}
      &\lim_{[0,1] \ni t \to s} \frac{F(\mu_t)-F(\mu_s)}{t-s} = \int_{\B \times \B} \la \rmD F_{\psi, \pphi}(\mu_s, \sfx^s(x_0,x_1)), x_1-x_0 \ra \d \mmu(x_0,x_1), \\ \label{eq:derivative1}
      &\lim_{[0,1] \ni t \to s} \frac{F(\nu_t)-F(\nu_s)}{t-s} = \int_{\B} \la \rmD F_{\psi, \pphi}(\nu_s, x), u(x) \ra \d \nu_s(x),
 \end{align}
for every $s \in [0,1]$. This is a simple consequence of the chain rule and the regularity of $\pphi$.
\end{remark}

\begin{remark}
  \label{rem:propDF}
   It is not difficult to check that
  \begin{equation}
    \label{eq:124}
    \rmD F_{\psi, \pphi}\text{ is continuous in }\prob(\B)\times \B
  \end{equation}
  with respect to the natural product (narrow and norm) topology.
  
  In principle
  $\rmD F_{\psi, \pphi}$ may depend on the choice of $N \in
  \N$, $\psi \in \rmC_b^1(\R^N)$ and $\uphi \in (\rmC_b^1(\B))^N$
  used to represent $F$. In Proposition \ref{prop:equality} we show
  that for every $\mu\in \prob_p(\B, \sfd_{\|\cdot\|})$ 
  the function $\rmD F_{\psi, \pphi}[\mu]$ is uniquely characterized in $\supp(\mu)$ so that $\rmD F_{\psi, \pphi}$ is
  uniquely characterized by $F$ in $\domG^p_\B:= \domG_\B \cap (\prob_p(\B, \sfd_{\|\cdot\|}) \times \B)$. We will be then able to remove the subscript in $\rmD F_{\psi, \pphi}$ and denote it simply by $\rmD_p F$.
\end{remark}

The following Lemma is proved in \cite{FSS22} and will be useful in the proof of Proposition \ref{prop:equality}.
\begin{lemma}\label{lem:limit}
   Let $Y$ be a Polish space
  and let $G:\prob(Y) \times Y \to [0, + \infty)$ be a bounded and
  continuous function. If
  $(\mu_n)_{n\in \N}$ is a sequence in $\prob(Y)$ narrowly converging (i.e.~converging in duality with continuous and bounded functions on $Y$)
  to $\mu$ as $n \to + \infty$, then
\[ \lim_{n\to\infty} \int_Y G(\mu_n, y) \d \mu_n(y) = \int_Y G(\mu, y) \d \mu(y).\]
\end{lemma}

The following proposition corresponds to \cite[Proposition 4.9]{FSS22} and the proof is quite similar but, because of a few differences, we still report it here.
\begin{proposition} \label{prop:equality} Let $(\B, \|\cdot\|)$ be a separable Banach space and let $F\in \ccyl{\prob(\B)}{\rmC_b^1(\B)}$; then, if $N \in \N$, $\psi \in \rmC_b^1(\R^N)$ and $\uphi \in (\rmC_b^1(\B))^N$ are such that $F=\psi\circ\lin\uphi$, we have
\[ \|\rmD F_{\psi, \pphi}[\mu]\|_{*, p', \mu} = \lip_{W_{p, \sfd_{\|\cdot\|}}} F (\mu) \quad \text{ for every } \mu \in \prob_p(\B, \sfd_{\|\cdot\|}).\]
In particular, $\|\rmD F_{\psi, \pphi}[\mu]\|_{*, p', \mu}$ does not depend on the choice of the
representation of $F$
and $\rmD F_{\psi, \pphi}$ just depends on $F$ on $\overline {\domG^p_\B}$ (cf.~Remark \ref{rem:propDF}).
\end{proposition}
\begin{proof} Let $\mu \in \prob_p(\B, \sfd_{\|\cdot\|})$ and let $(\mu'_n, \mu''_n) \in \prob_p(\B, \sfd_{\|\cdot\|})\times \prob_p(\B, \sfd_{\|\cdot\|})$ with $\mu'_n \ne \mu''_n$ be such that $(\mu'_n, \mu''_n) \to (\mu, \mu)$ in $W_{p, \sfd_{\|\cdot\|}}$ and
\[ \lim_n \frac{ \left |F(\mu'_n) - F(\mu''_n) \right |}{W_{p, \sfd_{\|\cdot\|}}(\mu'_n, \mu''_n)} = \lip_{W_{p, \sfd_{\|\cdot\|}}} F (\mu).\]
Let $(\mu_n^t)_{t \in [0,1]}$ be the curves defined as in \eqref{eq:curves} for plans $\mmu_n \in
\Gamma_{o,p, \sfd_{\|\cdot\|}}(\mu'_n, \mu''_n)$; we have
\begin{align*}
\left |F(\mu'_n) - F(\mu''_n) \right |&= \left |\int_0^1 \int_{\B\times \B} \la \rmD F_{\psi, \pphi}(\mu_n^t,\sfx^t(x_0,x_1)), x_1-x_0 \ra  \,\d \mmu_n(x_0,x_1) \,\d t \right | \\
 &\le \left ( \int_0^1 \int_{\B} \left \| \rmD F_{\psi, \pphi}(\mu_n^t,x) \right \|_*^{p'} \,\d \mu_n^t(x) \,\d t \right )^{\frac{1}{p'}} \\
 &\quad \,\,\left ( \int_0^1 \int_{\B} \left \|x_1-x_0 \right \|^p \,\d \mmu_n(x_0,x_1) \,\d t \right)^{\frac{1}{p}} \\
                                      &= W_{p, \sfd_{\|\cdot\|}}(\mu'_n, \mu''_n) \left ( \int_0^1 \int_{\B } \left \| \rmD F_{\psi, \pphi}(\mu_n^t,x) \right \|_*^{p'} \,\d  \mu_n^t(x)\,\d t \right )^{\frac{1}{p'}}, 
\end{align*}
where the first equality comes from \eqref{eq:derivative}. Dividing both sides by $W_{p, \sfd_{\|\cdot\|}}(\mu'_n, \mu''_n)$, we obtain
\[ \frac{\left |F(\mu'_n) - F(\mu''_n) \right |}{W_{p, \sfd_{\|\cdot\|}}(\mu'_n, \mu''_n)} \le \left ( \int_0^1 \int_{\B} \left \| \rmD F_{\psi, \pphi}(\mu_n^t,x) \right \|_*^{p'} \,\d  \mu_n^t(x) \,\d t \right )^{\frac{1}{p'}}.\]
Observe that $\mmu_n \to (\ii_{\B}, \ii_{\B})_\sharp
\mu$ in $\prob(\B\times\B)$
so that $\mu_n^t\to \mu$ in $\prob(\B)$ for every $t\in [0,1]$. 
We can pass to the limit as $n \to + \infty$ the above inequality using the dominated convergence Theorem and Lemma \ref{lem:limit} with
\[ G(\mu, x):= \left \| \rmD F_{\psi, \pphi}(\mu,x) \right \|_*^{p'}, \quad \mu \in \prob(\B), \, x\in \B.\]
We hence get 
\begin{align*}
    \lip_{W_{p, \sfd_{\|\cdot\|}}} F (\mu) &\le % \left ( \int_0^1 \int_{\B \times \B} \left | \rmD F[\mu]((1-t)x_0+t x_1) \right |^2 \,\d  (\text{id}_{\B}, \text{id}_{\B})_\sharp \mu (x_0,x_1) \,\d t \right )^{\frac{1}{2}}\\
     \left ( \int_0^1 \int_{\B \times \B} \left \| \rmD F_{\psi, \pphi}(\mu,x) \right \|_*^{p'} \,\d \mu (x) \,\d t \right )^{\frac{1}{p'}}=
                 \|\rmD F_{\psi, \pphi}[\mu]\|_{*, p', \mu}.
\end{align*}
To prove the other inequality we consider a countable dense subset $E:= \{x_n\}_n$ of the unit sphere in $\B$ and, for every $\eps>0$, the maps $N_\eps: \B^* \to \N$ and $j_{p',\eps}: \B^* \to \B$ defined as
\begin{align*}
    N_{\eps}(x^*)&:= \min \left \{ n \in \N : \la x^*, x_n \ra \ge \|x^*\|_* - \eps \right \}, \quad &&x^* \in \B^*,\\
    j_{p',\eps}(x^*)&:= \|x^*\|^{p'/p}x_{N_\eps(x^*)}, \quad &&x^* \in \B^*.
\end{align*}
It is not difficult to check that $j_{p',\eps}$ is measurable: if we define the sets
\begin{align*}
A_n&:= \left \{ x^* \in \B^* : \la x^*, x_n \ra \ge \|x^*\|_* - \eps \right \}, \quad && n \in \N,\\
U_1&:= A_1, \\
U_n &:= A_n \setminus \cup_{i=1}^{n-1} A_i, \quad && n \in \N, \, n >1,
\end{align*}
then $\{U_n\}_{n \in \N}$ is a countable measurable partition of $\B^*$ and $j_{p',\eps}$ can be written as $j_{p',\eps}(x^*)= \|x^*\|_*^{p'/p} x_n$ if $x \in U_n$, so that $j_{p',\eps}$ is measurable. Moreover it is obvious that
\begin{equation}\label{eq:propj}
    \|j_{p',\eps}(x^*)\|^p = \|x^*\|_*^{p'}, \quad \la j_{p',\eps}(x^*), x^* \ra \ge \|x^*\|_*^{p'}-\eps \|x^*\|_*^{p'/p}, \quad \text {for every } x^* \in \B^*.
\end{equation}
Let us now consider the maps $T,u_\eps: \B \to \B$ defined as
\[ T(x) := \rmD F_{\psi, \pphi}[\mu](x), \quad u_{\eps}(x):= j_{p',\eps}(T(x)),\quad  x \in \B,\]
and the curve $(\nu^\eps_t)_{t \in [0,1]}$ defined as in \eqref{eq:curves1} with $u_\eps$ as above (notice that $u_\eps \in L^p(\B, \mu; (\B, \|\cdot\|))$ since it is Borel measurable and $\|u_\eps(x)\|^p= \|T(x)\|_*^{p'} < C < +\infty$ for every $x \in \B$). By \eqref{eq:derivative1}, we get that
\[ \lim_{t \downarrow 0} \frac{F(\nu_t^\eps)-F(\mu) }{t} = \int_{\B} \la \rmD F_{\psi, \pphi}(\mu,x), u_\eps(x) \ra \,\d \mu(x) \ge \|T\|^{p'}_{L^{p'}(\B, \mu; (\B^*, \|\cdot\|_*))}-C^{1/p}\eps,\]
where we used \eqref{eq:propj}. Moreover
\[ \frac{W_{p, \sfd_{\|\cdot\|}}(\mu, \nu_t^\eps)}{t} \le \|u_\eps\|_{L^p(\B, \mu; (\B, \|\cdot\|))} = \|T\|^{p'/p}_{L^{p'}(\B, \mu; (\B^*, \|\cdot\|_*))} \quad \text{ for every } t \in (0,1].\]
Thus
\[\lip_{W_{p, \sfd_{\|\cdot\|}}} F (\mu) \ge \limsup_{t \downarrow 0} \frac{F(\nu_t^\eps)-F(\mu) }{W_{p, \sfd_{\|\cdot\|}}(\mu, \nu_t^\eps)} \ge \|T\|^{p'-p'/p}_{L^{p'}(\B, \mu; (\B^*, \|\cdot\|_*))}-\frac{C^{1/p}\eps}{\|T\|^{p'/p}_{L^{p'}(\B, \mu; (\B^*, \|\cdot\|_*))}}. \]
Passing to $\lim_{\eps \downarrow 0}$ we obtain the sought inequality and this concludes the proof.
\end{proof}
\begin{remark}\label{rem:separable}
Note that the inequality
\[ \lip_{W_{p, \sfd_{\|\cdot\|}}} F (\mu) \le \|\rmD F_{\psi, \pphi}[\mu]\|_{*, p', \mu} \quad \text{ for every } \mu \in \prob_p(\B, \sfd_{\|\cdot\|})\]
holds even if the Banach space $(\B, \|\cdot\||)$ is non-separable, since separability is not used in the first part of the proof.
\end{remark}
Thanks to Proposition \ref{prop:equality} the following definition is well posed.
\begin{definition} Let $(\B, \|\cdot\|)$ be a separable Banach space. For every $F\in \ccyl{\prob(\B)}{\rmC_b^1(\B)}$, we define
\begin{align}
\rmD_p F (\mu,x) &:= \rmD F_{\psi, \pphi} (\mu,x), \quad (\mu,x)\in \overline{\domG^p},\\
\| \rmD_p F[\mu] \|_{*, p', \mu}&:=\| \rmD F_{\psi, \pphi}[\mu] \|_{\mu, p'}, \quad \mu \in \prob_p(\B, \sfd_{\|\cdot\|}),
\end{align}
where $N \in \N$, $\psi \in \rmC_b^1(\R^N)$ and $\uphi \in (\rmC_b^1(\B))^N$ are such that $F=\psi\circ\lin\uphi$, and $\domG^p$ is as in Remark \ref{rem:propDF}.
\end{definition}

\subsection{The density result in \texorpdfstring{$(\R^d, \|\cdot\|)$}{R} for any norm}\label{subsec:wsspace}
In this whole subsection (apart from Theorem \ref{thm:main}) we focus again on a finite dimensional Banach space with a sufficiently regular norm. To this aim we fix a dimension $d \in \N$ and a norm $\|\cdot\|$ on $\R^d$ satisfying \eqref{eq:costf}. As we did in Subsection \ref{subsec:Kuseful}, we work on the complete and separable metric space $(\R^d, \sfd_{\|\cdot\|})$, where $\sfd_{\|\cdot\|}$ is the distance induced by $\|\cdot\|$, and the corresponding $(p, \sfd_{\|\cdot\|})$-Wasserstein space. To simplify the notation, as we did in Subsection \ref{subsec:Kuseful}, in this subsection, we will simply write $\prob_p(\R^d)$, $W_p$ and $\Gamma_{o,p}$, omitting the dependence on $\sfd_{\|\cdot\|}$. For the rest of this subsection $\mm$ is a positive and finite Borel measure on $\prob_p(\R^d)$.

Recall that for a bounded Lipschitz function
$F:\prob_p(\R^d)\to\R$ the pre-Cheeger energy (cf.~\eqref{eq:prec})  associated to $\mm$ is defined by
\begin{equation}
  \label{eq:52}
  \pCE_{q}(F)=\int_{\prob_p(\R^d)}\big(\lip_{W_p} F(\mu)\big)^{q}\,\d\mm(\mu).
\end{equation}
Thanks to Proposition \ref{prop:equality}, if $F$ is a cylinder
function in $\ccyl{\prob(\R^d)}{\rmC_b^1(\R^d)}$, we have a nice equivalent expression
\begin{equation}
  \label{eq:53}
  \pCE_{q}(F)=\int_{\prob_p(\R^d)}\|\rmD_p F[\mu]\|^{q}_{*,p', \mu}\,\d\mm(\mu)= \int_{\prob_p(\R^d)} \left ( \int_{\R^d} \|\rmD_p F(\mu,x)\|_*^{p'} \, \d \mu (x) \right )^{q/p'} \, \d \mm(\mu),
\end{equation}
where $\|\cdot\|_*$ is the dual norm induced by $\|\cdot\|$. Notice that $(\pCE_q)^{1/q}$ is not simply the $L^q$-$L^{p'}$ mixed norm of $\rmD_p F$ in a Bochner space, since the measures $\mu$ w.r.t.~the inner norm is computed varies, but rather the norm in the direct $L^q(\prob_p(\R^d),\mm)$-integral of the Banach spaces $L^{p'}(\B,\mu; (\B^*, \|\cdot\|_*))$ (see e.g. \cite{rms, grupre} or Section \ref{sec:refl}).

We adopt the notation $\AA:=\ccyl{\prob(\R^d)}{\rmC_b^1(\R^d)}$ and we devote this subsection to the proof of Theorem \ref{thm:main}. The following Lemmas are the obvious adaptation of \cite[Lemma 4.14, Lemma 4.15]{FSS22} and their proofs are the same, and thus omitted.

 \begin{lemma}
   \label{le:limsup-approximation}
   Let $F_n$ be a sequence of functions in $D^{1,q}(\prob_p(\R^d), W_p, \mm;\AA)\cap L^\infty(\prob_p(\R^d), \mm)$ such
   that $F_n$ and 
   $|\rmD F_n|_{\star,q,\AA}$ are uniformly bounded in every bounded set
   of $\prob_p(\R^d)$ and let $F,G$ be Borel functions in
   $L^{q}(\prob_p(\R^d),\mm)$, $G$ nonnegative.
   If
   \begin{equation}
     \label{eq:170}
     \lim_{n\to\infty}F_n(\mu)= F(\mu),\quad
     \limsup_{n\to\infty}|\rmD F_n|_{\star,q,\AA}(\mu)\le G(\mu)
     \quad\text{$\mm$-a.e.~in $\prob_p(\R^d)$},
   \end{equation}
   then $F\in H^{1,{q}}(\prob_p(\R^d), W_p, \mm;\AA)$ and $|\rmD F|_{\star,q,\AA}\le G$.
 \end{lemma}
 \begin{lemma}
   \label{le:speriamo-che-basti}
   Let $\phi\in \rmC^1(\R^d)$ be satisfying the growth conditions
      \begin{equation}
     \label{eq:186}
     \phi(x)\ge A\|x\|^p-B,\quad \|\nabla\phi(x)\|_*\le C(\|x\|^{p-1}+1)\quad\text{for
       every }x\in \R^d
   \end{equation}
   for given positive constants $A,B,C>0$
   and let $\zeta:\R\to\R$ be a $\rmC^1$ nondecreasing
   function whose derivative has compact support.
   Then
     the function
     $F(\mu):=\zeta\circ\lin\phi$
     is %$L$-
     Lipschitz in $\prob_p(\R^d)$,
     % with $L:=A^{-1/2}C$,
     it
     belongs to
   $H^{1,q}(\prob_p(\R^d), W_p, \mm;\AA)$, and
   \begin{equation}
     \label{eq:180}
     |\rmD F|_{\star,q,\AA}(\mu)\le \zeta'(\lin\phi(\mu))\Big(\int_{\R^d}\|\nabla\phi\|_*^{p'}\,\d\mu\Big)^{1/p'}.
   \end{equation}
 \end{lemma}
 Let $\kappa \in \rmC_c^{\infty}(\R^d)$ be such that $\supp{\kappa}=\overline{\rmB(0,1)}$ (here $\rmB(0,1)$ is the unit $d$-dimensional Euclidean ball) , $\kappa(x) \ge 0$ for every $x \in \R^d$, $\kappa(x)>0$ for every $x \in \rmB(0,1)$, $\int_{\R^d} \kappa \d \mathcal{L}^d=1$ and $\kappa(-x)=\kappa(x)$ for every $x \in \R^d$. Let us define, for every $0<\eps<1$, the standard mollifiers
\[ \kappa_\eps (x) := \frac{1}{\eps^d} \kappa (x/\eps), \quad x \in \R^d.\]
Given $\sigma \in \prbt$ and $0<\eps<1$, we define
\[
\sigma_\eps := \sigma \ast \kappa_\eps.
  \]
Notice that $\sigma_\eps \in \prob_p^r(\R^d)$ and $W_p(\sigma_\eps, \sigma) \to 0$ as $\eps \downarrow 0$. Moreover, if $\sigma, \sigma' \in \prbt$, we have
\begin{equation}\label{eq:ineqconv}
    W_p(\sigma_\eps, \sigma'_\eps) \le W_p(\sigma, \sigma') \quad \text{ for every } 0<\eps<1
\end{equation}
and it is easy to check that, if we set
\begin{equation}\label{eq:kappamom}
    C_\eps := \rsqm {\kappa_\eps \mathcal{L}^d},
\end{equation}
then we have
\begin{equation}\label{eq:contrmom}
\rsqm{\mu_\eps} \le \rsqm{\mu} + C_\eps \quad \text{ for every } 0<\eps<1.
\end{equation}
\begin{definition} Let $0<\eps < 1$ and $\nu \in \prob_p(\R^d)$. We define the continuous functions $W_\nu, W_\nu^\eps, F_\nu^\eps : \prbt \to \R$ as
\begin{align*}
W_\nu(\mu):= W_p(\mu, \nu), \quad W_\nu^{\eps}(\mu) := W_{\nu}(\mu_\eps), \quad F^\eps_\nu (\mu) := \frac{1}{p}(W_\nu^{\eps}(\mu))^p, \quad \mu \in \prbt.
 \end{align*}
\end{definition}
The proof of the following proposition follows the one of \cite[Proposition 4.17]{FSS22} but since the exponent $p$ may be different from $2$, the estimates are more complicated and thus reported in full.
\begin{proposition}\label{prop:fund} Let $0<\eps < 1$, $\delta, R>0$, let $\nu \in \prob_p^r(\R^d)$ be such that $\supp{\nu}=\overline{\rmB_{\|\cdot\|}(0,R)}$ and $\nu \mres \rmB_{\|\cdot\|}(0,R) \ge \delta \mathcal{L}^d \mres \rmB_{\|\cdot\|}(0,R)$, and let $\zeta:\R\to\R$ be a $\rmC^1$ nondecreasing function whose derivative has compact support. Then
  \begin{equation}
    \label{eq:174}
  |\rmD(\zeta \circ F_\nu^\eps)|_{\star,q, \AA}(\mu) \le
  \zeta'(F^\eps_\nu(\mu))W_p^{p-1}(\nu, \mu_\eps) \quad
    \text{ for } \mm \text{-a.e. }\mu \in \prbt.
  \end{equation}
\end{proposition}
\begin{proof}
  Let $\mathcal{G}:=\{\mu^h\}_{h \in \N}$ be a dense and countable set
  in $\prbt$ and let us set, for every $h \in \N$, $\varphi_h :=
  \Phi(\nu, \mu^h_\eps)$, $\varphi_h^* :=
  \Phi^*
  (\nu, \mu^h_\eps)$ (see Theorem \ref{thm:ot}), 
\[ a_h:= \int_{\rmB_{\|\cdot\|}(0,R)}\varphi_h \d \nu, \quad u_h:=\varphi_h^* +a_h, \quad G_k(\mu):= \max_{1 \le h \le k} \int_{\R^d} u_h \d \mu_\eps, \quad
  \mu \in \prbt, \quad k \in \N. \]
Notice that, by \eqref{eq:phiest1}, $u_h \in L^1(\R^d, \mu_\eps)$ for every $\mu \in \prob_p(\R^d)$.

\emph{Claim 1}. It holds
\[ \lim_{k \to + \infty} G_k(\mu) = %\GGG \frac 12
  F_\nu^{\eps}(\mu) \quad \text{ for every } \mu \in \prbt.\]
\emph{Proof of claim 1}. Since $G_{k+1}(\mu) \ge G_{k}(\mu)$ for every $\mu \in \prbt$, we have that 
\[ \lim_{k \to + \infty} G_k(\mu) = \sup_k G_k(\mu) = \sup_h \int_{\R^d} u_h \d \mu_\eps \quad \text{ for every } \mu \in \prbt.\]
By the definition of $\varphi_h$ and $\varphi_h^*$ we have, for every $\mu \in \prbt$ and $h \in \N$, that
\begin{align*}
    \int_{\R^d} u_h \d \mu_\eps &= \int_{\R^d}\varphi_h^* \d \mu_\eps + \int_{\rmB_{\|\cdot\|}(0,R)} \varphi_h(y) \d \nu\\
                                & \le \frac{1}{p} W_p^p(\mu_\eps, \nu)\\
  % \\          &
                &= %\GGG \frac 12
                                  F_\nu^{\eps}(\mu). 
\end{align*}
This proves that $\sup_k G_k(\mu) \le
F_\nu^{\eps}(\mu)$ for every $\mu \in \prbt$. Clearly, if $\mu \in \mathcal{G}$, this is an equality. By \eqref{eq:w2lipphi} there exists a constant $D_{p,R}$ such that, for every $h \in \N$ and $\mu, \mu' \in \prob_p(\R^d)$, it holds
\begin{align*}
    \int_{\R^d}u_h \d \mu_\eps - \int_{\R^d} u_h \d \mu'_\eps &=  \int_{\R^d} \varphi_h^* \d (\mu_\eps-\mu'_\eps) \\
    & \le D_{p,R} W_p(\mu_\eps, \mu'_\eps)(1+\rsqm{\mu_\eps}+\rsqm{\mu'_\eps})\\
    & \le D_{p,R} W_p(\mu, \mu') (1+2C_\eps + \rsqm{\mu}+\rsqm{\mu'}),
\end{align*}
where we used \eqref{eq:ineqconv} and \eqref{eq:contrmom}. We hence deduce that for every $k \in \N$
\begin{equation}\label{eq:12}
    \left |G_k(\mu)-G_k(\mu') \right | \le D_{p,R} W_p(\mu, \mu') (1+2C_\eps + \rsqm{\mu}+\rsqm{\mu'}) \quad \text{ for every } \mu, \mu' \in \prbt.
\end{equation}
Choosing $\mu' \in \mathcal{G}$ and passing to the limit as $k \to + \infty$ we get from \eqref{eq:12} that
\[ \left | \lim_{k \to + \infty} G_k(\mu)-F_\nu^\eps(\mu') \right | \le D_{p,R} W_p(\mu, \mu') (1+2C_\eps + \rsqm{\mu}+\rsqm{\mu'})  \text{ for every } \mu \in \prbt, \, \mu' \in \mathcal{G}.\]
Using the density of $\mathcal{G}$ and the continuity of $\mu' \mapsto F_\nu^{\eps}(\mu')$ we deduce that 
\[ \lim_{k \to + \infty} G_k(\mu) =
  F_\nu^\eps(\mu) \quad \text{ for every } \mu \in \prbt\]
proving the first claim.

\emph{Claim 2}.
If $H_k:= \zeta\circ G_k$ and $u_{h,\eps}:=u_h\ast \kappa_\eps$, it holds
\[ | \rmD H_k |_{\star, q, \AA}^{p'}(\mu) \le
  \big(\zeta'(G_k(\mu))\big)^{p'}\int_{\R^d}\left \| \nabla u_{h,\eps} \right \|_*^{p'}  \d \mu =
  \big(\zeta'(G_k(\mu))\big)^{p'}\int_{\R^d}\left \|\nabla (\varphi_h^* \ast \kappa_\eps) \right \|_*^{p'} \d \mu, \]
for $\mm$-a.e.~$\mu \in B_h^k$, where $B_h^k := \{ \mu \in \prbt : G_k(\mu) = \int_{\R^d} u_h \d \mu_\eps \}$, $h \in \{1, \dots, k\}$.\\
\emph{Proof of claim 2}. For every $h \in \N$, \eqref{eq:lowerboundphi} and \eqref{eq:phiest2} (also using \eqref{eq:ineqtt}) yield
\begin{equation}
  \label{eq:193}
  u_{h,\eps}(x)\ge \frac{1}{4p}\|x\|^p +a_h - A_{p,R,\eps}, \quad 
  \|\nabla u_{h,\eps}(x)\|_*\le A_{p,R,\eps}(1+\|x\|^{p-1}) \quad \text{ for every } x \in \R^d,
\end{equation}
where $A_{p,R,\eps}>0$ is a constant depending only on $p,R,\eps$. Since  the map $\ell_h: \prbt \to \R$ defined as $\ell_h(\mu):=
\int_{\R^d} u_h \d \mu_\eps $ satisfies
\[ \ell_h(\mu)= \int_{\R^d} (u_h \ast \kappa_\eps) \d \mu=\lin{u_{h,\eps}}(\mu), \quad \mu \in \prbt,\]
Lemma \ref{le:speriamo-che-basti} and the
above estimates yield
%is cylinder (cf.~\eqref{eq:monocyl}). Thus, from Proposition \ref{prop:equality} and the definition of relaxed gradient, we get
\[ |\rmD(\zeta\circ \ell_h)|_{\star, q, \AA}(\mu) \le
  \zeta'(\ell_h(\mu))\Big(\int_{\R^d} \left \| \nabla
    % (u_h \ast    \kappa_\eps)
    u_{h,\eps}
  \right \|_*^{p'} \d \mu \Big)^{1/p'}\quad
  \text{for $\mm$-a.e. } \mu \in \prbt.\]
Since $H_k$ can be written as
\[ H_k(\mu) = \max_{1 \le h \le k} (\zeta\circ \ell_h)(\mu), \quad \mu \in \prbt,\]
we can apply Theorem \ref{thm:omnibus} (4) and conclude the proof of
the second claim.

\emph{Claim 3}. For every $R>0$ there exists a constant $C>0$ independent of $h$  such
that
\begin{equation}
  \label{eq:192}
  \Big(\int_{\R^d} \left \| \nabla u_{h,\eps}
  \right \|_*^{p'} \d \mu\Big)^{1/p'} \le C\quad\text{whenever }\rsqm\mu\le R.
\end{equation}
\emph{Proof of Claim 3}.
It is sufficient to use \eqref{eq:193}.

\emph{Claim 4}. Let $(h_n)_n \subset \N$ be an increasing sequence and
let $\mu \in \prbt$. If $\lim_n \int_{\R^d} u_{h_n} \d \mu_\eps
=% \GGG\frac 12 
F_\nu^\eps(\mu)$, then 
\[ \limsup_n \int_{\R^d} \left \|\nabla ( \varphi_{h_n}^* \ast \kappa_\eps ) \right \|_*^{p'} \d \mu  \le W_p^{p}(\nu, \mu_\eps).\]
\quad \\
\emph{Proof of claim 4}. Since
\[ \left \|\nabla (\varphi_{h_n}^* \ast \kappa_\eps)(x) \right \|_*^{p'} \le \left (\left \|\nabla \varphi_{h_n}^* \right \|_*^{p'} \ast \kappa_\eps \right )(x) \quad \text{ for every } x \in \R^d,\]
we get that
\begin{align*}
                                  \int_{\R^d} \left \|\nabla
                                  (\varphi_{h_n}^* \ast \kappa_\eps)
                                  \right \|_*^{p'} \d \mu
                                   & \le \int_{\R^d} \left ( \left \|\nabla \varphi_{h_n}^* \right \|_*^{p'} \ast \kappa_\eps \right ) \d \mu \\
 & = \int_{\R^d} \left \|\nabla \varphi_{h_n}^* \right \|_*^{p'}\d \mu_{\eps}.
\end{align*}
It is then enough to prove that 
\begin{equation}\label{eq:easier}
\limsup_n\int_{\R^d} \left \|\nabla \varphi_{h_n}^* \right \|_*^{p'}\d \mu_{\eps} \le W_p^{p}(\nu, \mu_\eps).
\end{equation}
Let us set $\phi_n:= \varphi_{h_n}$ and $\psi_n:=\varphi_{h_n}^*$ and let us extract a (unrelabeled) subsequence such that the $\limsup$ in the statement of the claim is achieved as a limit. Since $\phi_n- \frac{1}{p}\|x\|^p \le 0$ for every $x \in \R^d$ and $\nu \mres \rmB_{\|\cdot\|}(0,R) \ge \delta \mathcal{L}^d \mres \rmB_{\|\cdot\|}(0,R)$, we have
\begin{align}\label{eq:unpezzo} \begin{split}
    \int_{\rmB_{\|\cdot\|}(0,R)} \phi_n \, \d \mathcal{L}^d &= \int_{\rmB_{\|\cdot\|}(0,R)} \left (\phi_n(x)- \frac{1}{p}\|x\|^p \right ) \, \d \mathcal{L}^d(x) + \frac{1}{p} \sqm{\mathcal{L}^d \mres \rmB_{\|\cdot\|}(0,R)} \\
    & \ge \frac{1}{\delta} \int_{\rmB_{\|\cdot\|}(0,R)} \phi_n \, \d \nu - \frac{1}{p} \sqm{\nu \mres \rmB_{\|\cdot\|}(0,R)} + \frac{1}{p} \sqm{\mathcal{L}^d \mres \rmB_{\|\cdot\|}(0,R)}.\end{split}
\end{align}
By the convergence of $\int_{\R^d} u_{h_n} \d \mu_\eps$ to $F_{\nu}^{\eps}(\mu)$,  we can find $N \in \N$ such that 
\[ \int_{\rmB_{\|\cdot\|}(0,R)} \phi_n \, \d \nu  \ge -1 - \int_{\R^d} \psi_n \, \d \mu_{\eps} \ge -1 - K_{p,R} \left ( 1+ \sqm{\mu_{\eps}} \right) \quad \text{ for every } n \ge N,\]
where the last inequality comes from \eqref{eq:phiest1}. Combining the above inequality with \eqref{eq:unpezzo} we get that there exists some $I>0$ such that 
\begin{equation}
    \label{eq:unboundcond}
    \int_{\rmB_{\|\cdot\|}(0,R)} \phi_n \, \d \mathcal{L}^d \ge - I \quad \text{ for every } n \in \N.
\end{equation}
We can thus apply Proposition \ref{prop:oldcl4567} and obtain the existence of a subsequence $j \mapsto n(j)$ and two locally Lipschitz functions $\phi$ and $\psi$ such that $\phi_{n(j)} \to \phi$ locally uniformly in $\rmB_{\|\cdot\|}(0,R)$, $\psi_{n(j)} \to \psi$ locally uniformly in $\R^d$ and \eqref{eq:convgrad} holds. Being the inequality
\[ \phi_{n(j)}(x) + \psi_{n(j)}(y) \le \frac{1}{p}\|x-y\|^p \quad \text{ for every } (x,y) \in \rmB_{\|\cdot\|}(0,R) \times \R^d\]
satisfied for every $j \in \N$, we can pass to the limit and obtain point (i) in Theorem \ref{thm:ot} for the pair $(\phi, \psi)$. Since $\phi_n(x) \le \frac{1}{p}\|x\|^p$ for every $x \in \R^d$ we get by Fatou's lemma that
\begin{equation}\label{eq:unlim}
   -I \le \limsup_j \int_{\rmB_{\|\cdot\|}(0,R)} \phi_{n(j)} \, \d \nu \le \int_{\rmB_{\|\cdot\|}(0,R)} \limsup_j \phi_{n(j)}(x) \, \d \nu(x) = \int_{\rmB_{\|\cdot\|}(0,R)} \phi \, \d \nu \le \frac{1}{p} \sqm{\nu}. 
\end{equation}
This in particular gives that $\phi \in L^1(\rmB_{\|\cdot\|}(0,R), \nu)$. By \eqref{eq:phiest1} we have that $|\psi(y)| \le K_{p,R}(1+\|y\|^p)$ so that $\psi \in L^1(\R^d, \mu)$ and we can also apply the dominated convergence theorem and obtain that
\begin{equation}\label{eq:unlim1}
    \lim_j \int_{\R^d} \psi_{n(j)} \, \d \mu_{\eps} = \int_{\R^d} \psi \, \d \mu_{\eps}.
\end{equation}
Combining \eqref{eq:unlim} with \eqref{eq:unlim1} and using that $\lim_n \int_{\R^d} u_{h_n} \, \d \mu_\eps = F_\nu^\eps(\mu)$, we get point (ii) in Theorem \ref{thm:ot} for the pair $(\phi, \psi)$, so that, by the first part of Theorem \ref{thm:ot}, we conclude that 
\[ \int_{\R^d} \|\nabla \psi \|_*^{p'} \, \d \mu_\eps = W_p^{p}(\nu, \mu_{\eps}).\]
Then, since by \eqref{eq:phiest2} there exists a constant $K_{p,R}>0$ such that
\[ \|\nabla \psi_{n(j)}(y)\|_*^{p'} \le K_{p,R}^{p'} (1+\|y\|^{p-1})^{p'} \le 2^{p'} K_{p,R}^{p'}(1+\|y\|^p) \in L^1(\R^d, \mu) \quad \text{ for $\mathcal{L}^d$-a.e.~$y \in \R^d$},\]
and for every $j \in \N$, we can use the dominated convergence theorem to conclude that
\[  \limsup_n \int_{\R^d} \|\nabla \psi_n \|_*^{p'} \d \mu_{\eps} = \lim_j \int_{\R^d} \|\nabla \psi_{n(j)} \|_*^{p'} \d \mu_{\eps} = \int_{\R^d} \|\nabla \psi\|_*^{p'} \, \d \mu_{\eps} = W_p^p(\nu, \mu_{\eps}).\]
This concludes the proof of the fourth claim.

\emph{Claim 5}. It holds
\[ \limsup_k |\rmD H_k|_{\star, q, \AA}(\mu) \le \zeta'(F^\eps_\nu(\mu))
  W_p^{p/p'}(\nu, \mu_\eps) \quad \text{ for $\mm$-a.e. } \mu \in \prbt.\]
\emph{Proof of claim 5}. Let $B \subset \prbt$ be defined as
\[ B:= \bigcap_k \bigcup_{h=1}^k A_h^k,\]
where $A_h^k$ is the full $\mm$-measure subset of $B_h^k$ where claim 2 holds. Notice that $B$ has full $\mm$-measure. Let $\mu \in B$ be fixed and let us pick an increasing sequence $k \mapsto h_k$ such that 
\[ G_k(\mu) = \int_{\R^d} u_{h_k} \d \mu_\eps.\] 
By claim 1 we know that $G_n(\mu) \to F_\nu^\eps(\mu)$ so that we can apply claim 4 and conclude that 
\[ (\zeta'(F^\eps_\nu(\mu)))^{p'}   W_p^{p}(\nu, \mu_\eps) \ge
  \limsup_{k} (\zeta'(G_k(\mu)))^{p'}\int_{\R^d} \left \|\nabla ( \varphi_{h_k}^* \ast \kappa_\eps ) \right \|_*^{p'} \d \mu.\]
By claim 2, the right hand side is greater than $\limsup_k |\rmD H_k|_{\star, q, \AA}^{p'}(\mu)$; this concludes the proof of the fifth claim.

Eventually, we observe that by Claim 1
\begin{equation}
  \label{eq:172}
   (\zeta \circ F_\nu^\eps)(\mu) =\lim_{k\to\infty} (\zeta \circ G_k)(\mu)=\lim_{k\to\infty}H_k(\mu) \quad \text{ for every } \mu \in \prbt.
\end{equation}
Moreover, it is clear that 
\[ (\zeta'(F^\eps_\nu(\mu)))^{p'}   W_p^{p}(\nu, \mu_\eps)\]
is uniformly bounded. We can then combine the expression of Claim 2, the uniform estimate of
Claim 3, the limit of Claim 5 with Lemma \ref{le:limsup-approximation} to get \eqref{eq:174}.
% \begin{equation}
%   \label{eq:173}
%   |\rmD \vartheta(2G_k)|_{\star,p,\AA}(\mu)=
%   2\vartheta'(2G_k)|\rmD G_k|_{\star,p,\AA}(\mu)
% \end{equation}
% Claim 4 then yields \eqref{eq:174}.
% \footnote{Controllare che si pu\`o usare il limsup puntuale dei relaxed gradients}
\end{proof}
Precisely as in \cite[Corollary 4.18]{FSS22} we get the following corollary (its proof can be easily adapted and thus omitted).
\begin{corollary}
  \label{cor:bound} Let $\nu \in \prob_p^r(\R^d)$ be such that $\supp{\nu}=\overline{\rmB_{\|\cdot\|}(0,R)}$ and $\nu \mres \rmB_{\|\cdot\|}(0,R) \ge \delta \mathcal{L}^d \mres \rmB_{\|\cdot\|}(0,R)$ for some $\delta, R>0$. Then 
  \begin{equation}
|\rmD W_\nu|_{\star, q, \AA}(\mu) \le 1\quad \text{for $\mm$-a.e. } \mu \in
  \prbt.\label{eq:196}
\end{equation}
\end{corollary}
The following theorem provides the main result of this subsection. Recall the notation for Wasserstein spaces as in Section \ref{sec:Wasserstein} and the definition of density in $q$-energy as in Definition \ref{def:density}.
\begin{theorem}
  \label{thm:main} Let $d \in \N$,  let $p,q \in (1,+\infty)$ be (not necessarily conjugate) fixed exponents  and let $\|\cdot\|$ be any norm on $\R^d$. Then the algebra $\ccyl{\prob(\R^d)}{\rmC_b^1(\R^d)}$ is dense in $q$-energy in $D^{1,p}(\prob_p(\R^d, \sfd_{\|\cdot\|}), W_{p, \sfd_{\|\cdot\|}}, \mm)$, where $\sfd_{\|\cdot\|}$ is the distance induced by the norm $\|\cdot\|$.
\end{theorem}
\begin{proof} Let us first suppose that $\|\cdot\|$ is a norm as in \eqref{eq:costf}. Then the density result follows by Corollary \ref{cor:bound} and Theorem \ref{theo:startingpoint}, since the set of $\nu \in \prob_p^r(\R^d)$ such that there exist $R,\delta>0$ for which $\supp(\nu)= \overline{\rmB_{\|\cdot\|}(0,R)}$ and $\nu \mres \rmB_{\|\cdot\|}(0,R) \ge \delta \Leb d\mres\rmB_{\|\cdot\|}(0,R)$ is dense in $\prob_p(\R^d)$. This can be seen defining for every $\eps \in (0,1)$ the measures 
\[ \hat{\nu}_\eps := \frac{\nu_\eps \mres \rmB_{\|\cdot\|}(0,1/\eps) +
                      \eps^{d+p+1} \mathcal{L}^d \mres
                      \rmB_{\|\cdot\|}(0,1/\eps)}{\nu_\eps(\rmB_{\|\cdot\|}(0,1/\eps)) +
                      \eps^{d+p+1} \mathcal{L}^d(\rmB_{\|\cdot\|}(0,1/\eps))},\]
so that $\hat{\nu}_\eps \to \nu$ in $\prob_p(\R^d)$, $\supp(\nu) = \overline{\rmB_{\|\cdot\|}(0, 1/\eps)}$ and $\hat \nu^\eps \mres \rmB_{\|\cdot\|}(0,1/\eps) \ge \frac{\eps^{d+p+1}}{1+\eps^{d+p+1} \omega_{\eps,d}}\Leb d\mres\rmB_{\|\cdot\|}(0,1/\eps)$, where $\omega_{\eps,d}$ is the $\mathcal{L}^d$-measure of $\rmB_{\|\cdot\|}(0,1/\eps)$. 

Let us suppose now that $\|\cdot\|$ is any norm on $\R^d$. We can construct a sequence of norms $(\|\cdot\|_{k})_k$ on $\R^d$ satisfying \eqref{eq:costf} and such that
\[ \|x\|_k \uparrow \|x\| \quad \text{ for every } x \in \R^d.\]
For example we can take as $\|\cdot\|_k$ the norm whose unit ball is given by the set 
\[ C_k:= \left \{ x \in \R^d : \frac{\|x\|+|x|^2/k}{1+\eta^2/k} \right \} + \rmB_{\|\cdot\|}(0,1/k),\]
where $|\cdot|$ is the Euclidean norm on $\R^d$ and $\eta>0$ is a constant such that $|x| \le \eta \|x\|$ for every $x \in \R^d$. Notice that by construction $C_k$ is strictly convex, it contains the $\|\cdot\|$-unit ball and it satisfies the $1/k$-ball condition so that  its boundary is $\rmC^{1,1}$ (see \cite[Definition 1.1 and Theorem 1.8]{dalphin}). Also by construction we have that $\|\cdot\|_k \uparrow \|\cdot\|$ as $k \to +\infty$.

Let us denote by $W_{p,k}$ the $(p, \sfd_{\|\cdot\|_k})$-Wasserstein distance on $\prob_p(\R^d)=\prob_p(\R^d, \sfd_{\|\cdot\|})=\prob_p(\R^d, \sfd_{\|\cdot\|_k})$. We have that
\[ W_{p,k}(\mu, \nu) \uparrow W_{p, \sfd_{\|\cdot\|}}(\mu, \nu) \quad \text{ for every } \mu, \nu \in \prob_p(\R^d)\]
and, being the norms $\|\cdot\|$ and $\|\cdot\|_k$ equivalent, then the distances $W_{p,k}$ and $W_{p, \sfd_{\|\cdot\|}}$ induce the same topology on $\prob_p(\R^d)$. Let $F \in \Lipb(\prob_p(\R^d), W_{p,\sfd_{\|\cdot\|}})$; by \cite[Proposition 3.3]{pasqualetto} (see also \cite[Theorem 9.1]{aes}), we can find a subsequence $(k_n)_n$ such that, for every $n \in \N$, there exists $F_n \in \Lipb(\prob_p(\R^d), W_{p,k_n})$ such that\footnote{Here we are adding a subscript both in the notation for the Cheeger and the pre-Cheeger energy to specify the distance w.r.t.~which they are computed.}
\[ \|F-F_n\|_{L^q(\prob_p(\R^d), \mm)}< 1/n, \quad \pCE_{q,W_{p,k_n}}(F_n) \le \pCE_{q,W_{p, \sfd_{\|\cdot\|}}}(F) + 1/n.\]
Since $\AA= \ccyl{\prob(\R^d)}{\rmC_b^1(\R^d)}$ is dense in $q$-energy in $D^{1,p}(\prob_p(\R^d, \sfd_{\|\cdot\|_{k_n}}), W_{p,{k_n}}, \mm)$ by the first part of the proof, we have that
\[ \CE_{q,\AA, W_{p,{\sfd_{\|\cdot\|}}}}(F_{n}) \le \CE_{q,\AA, W_{p,{k_n}}}(F_{n}) \le \pCE_{q,W_{p,k_n}}(F_n) \le \pCE_{q,W_{p, \sfd_{\|\cdot\|}}}(F) + 1/n \text{ for every } n \in \N.\]
Passing to the $\liminf_n$ we get that 
\[ \CE_{q,\AA, W_{p,{\sfd_{\|\cdot\|}}}}(F) \le \pCE_{q,W_{p, \sfd_{\|\cdot\|}}}(F),\]
which gives the desired density, since this entails equality of the Cheeger energies and thus of the minimal relaxed gradients (see also Remark \ref{rem:dense2}).
\end{proof}
Arguing precisely as in \cite[Proposition 4.19]{FSS22}, it is not difficult to see that we can obtain the density result also for smaller algebras of functions on $\R^d$.
\begin{proposition}\label{prop:morealg} Let $d \in \N$, $p,q \in (1,+\infty)$ and let $\|\cdot\|$ be any norm on $\R^d$. Let $\EE \subset \rmC_b^1(\R^d)$ be an algebra of functions such that for every $f \in \rmC_b^1(\R^d)$ there exists a sequence $(f_n)_n \subset \EE$ such that
\[ \sup_{\R^d}|f_n|+ |\nabla f_n| < + \infty, \quad \lim_{n \to + \infty} \int_{\R^d} \left ( |f_n-f| + |\nabla f_n - \nabla f| \right ) \d \mu =0 \text{ for $\mm$-a.e.~}\mu \in \prob_p(\R^d, \sfd_{\|\cdot\|}).\]
Then the algebra $\ccyl{\prob(\R^d)}{\EE}$ is dense in $q$-energy in $D^{1,p}(\prob_p(\R^d, \sfd_{\|\cdot\|}), W_{p, \sfd_{\|\cdot\|}}, \mm)$.
\end{proposition}
\begin{remark} A possible choice for $\EE$ in Proposition \ref{prop:morealg} is
\begin{equation}\label{eq:poly}
\mathscr{P}_{d,C} := \left \{ \varphi \in \rmC_b^1(\R^d) : \varphi \bigr |_{[-C,C]^d} \text{ is a polynomial } \right \}
\end{equation}
where $C>0$ is any positive constant.
\end{remark}

\subsection{The density result in \texorpdfstring{$(X, \sfd)$}{X}}\label{sec:densxd}
In this subsection we prove the density in energy of suitable cylinder functions for an arbitrary separable metric space $(X,\sfd)$. We consider on $(X,\sfd)$ a sequence $\Phi:= (\phi_k)_k \subset \Lipb(X,\sfd)$ such that
\begin{equation}\label{eq:compseq}
    \sfd(x,y)= \sup_{\phi \in \Phi} |\phi(x)-\phi(y)| = \sup_k |\phi_k(x)-\phi_k(y)| \quad \text{ for all } x,y \in (X,\sfd),
\end{equation}
and we define
\begin{equation}\label{eq:phid}
\Phi(d):= \max_{1 \le k \le d} \|\phi_k\|_\infty, \quad d \in \N.
\end{equation}
We will have to embed $(X, \sfd)$ into $\ell^\infty(\N)$ and, for this reason, we need to fix some related notation.

We consider the family of maps $E:= \{e_i\}_{i \in \N} \subset (\ell^\infty(\N))^*$, where for every $i \in \N$, $e_i$ is defined as
\[ e_i:\ell^\infty(\N) \to \R, \quad e_i((x_k)_k)=x_i \quad \text{ for every } (x_k)_k \in \ell^\infty(\N).\]
We consider
the collection of maps 
$\pi_d:\ell^{\infty}(\N) \to \R^d$, $d\in \N$, given by
\begin{equation}
  \label{eq:91}
  \pi^d((x_k)_k):=(e_1((x_k)_k), \dots, e_d((x_k)_k)=(x_1, \dots, x_d), \quad \text{ for every } (x_k)_k \in \ell^\infty(\N).
\end{equation}
The adjoint map $\pi^{d*}:\R^d\to(\ell^\infty(\N))^*$ is given by
\begin{equation}
  \label{eq:103}
  \pi^{d*}(y_1,\cdots,y_d):= \sum_{i=1}^d y_i e_i, \quad \text{ for every } (y_1, \dots, y_d) \in \R^d.
\end{equation}
We say that a function $\phi:\ell^{\infty}(\N)\to\R$ belongs to $\rmC^1_b(\ell^{\infty}(\N),E,\Phi)$ if it can be
written as
\begin{equation}
  \label{eq:95}
  \phi:=\varphi\circ \pi^d\quad\text{for some }d\in \N,\ \varphi\in \mathscr{P}_{d,\Phi(d)},
\end{equation}
where $\mathscr{P}_{d,\Phi(d)}$ is as in \eqref{eq:poly} and $\Phi(d)$ is as in \eqref{eq:phid}.
Clearly $\phi\in\rmC^1_b(\ell^{\infty}(\N))$ and its gradient $\nabla \phi$ can be written as
\begin{equation}
  \label{eq:96}
  \nabla\phi=\pi^{d*}\circ \nabla\varphi\circ\pi^d,\quad
  \nabla\phi(x)= 
  \sum_{i=1}^d \partial_i\varphi(\pi^d(x)) e_i, \quad x \in \ell^{\infty}(\N).
\end{equation}
On $\prob(\ell^{\infty}(\N))$, we consider the algebra of cylinder functions $\AA':=\ccyl{\prob(\ell^\infty(\N))}{\rmC^1_b(\ell^{\infty}(\N),E, \Phi)}$ (recall Definition \ref{def:cyl}). For every $F\in \AA'$ we can find $N\in \N$, a function
$\psi\in \rmC^1_b(\R^N)$ and functions $\phi_n  \in
\rmC^1_b(\ell^{\infty}(\N),E, \Phi)$,  $n=1,\cdots,N$,   such that
\begin{equation}
  \label{eq:93}
  F(\mu) =(\psi \circ \lin \pphi)(\mu) \quad \text{ for every } \mu \in \prob(\ell^\infty(\N)).
\end{equation}
It is also easy to check that a function $F$ belongs to
$\AA'$
if and only if there exists $d\in \N$ and $\tilde F\in
\ccyl{\prob(\R^d)}{\mathscr{P}_{d,\Phi(d)}}$ such that
\begin{equation}
  \label{eq:102}
  F(\mu)=\tilde F(\pi^d_\sharp(\mu))\quad\text{for every }\mu\in \prob(\ell^\infty(\N)),
\end{equation}
so that, by Proposition \ref{prop:equality}, we have
\begin{align}
\rmD_p F(\mu,x)&=\pi^{d*}\Big(\rmD_p\tilde
  F(\pi^d_\sharp\mu,\pi^d(x))\Big),\\ \label{eq:104}
   \|\rmD_p F[\mu]\|_{*,p',\mu}&=\|\rmD_p \tilde F[\pi^d_\sharp\mu]\|_{*,p',\pi^d_\sharp\mu} = \lip_{W_{p,\sfd_{\|\cdot\|_{d, \infty}}}} \tilde F (\pi^d_\sharp\mu),
\end{align}
where we are using $\|\cdot\|_{*, \infty}$, the dual norm of $\|\cdot\|_\infty$ (the norm in $\ell^\infty(\N)$), for $\rmD_p F[\mu](x)$, $x \in \ell^\infty(\N)$, and the norm $\|y\|_{*,d}:=\|\pi^{d,*}(y)\|_{*,\infty}$ for $\rmD_p \tilde{F}[\pi_\sharp^d \mu](y)$, $y \in \R^d$. Notice that $\|\cdot\|_{*,d}$ is the $1$-norm on $\R^d$; in particular, the dual norm $\|\cdot\|_{d, \infty}$ on $\R^d$ of $\|\cdot\|_{*,d}$ is the $\infty$-norm on $\R^d$. The proof of the following theorem combines the finite dimensional projections technique of \cite[Theorem 6.4]{FSS22} with a standard embedding strategy.
\begin{theorem}
  \label{thm:main3} Let $(X,\sfd)$ be a complete and separable metric space and let $\{\phi_k\}_{k \in \N} \subset \Lipb(X,\sfd)$ be a countable set of functions such that 
  \[ \sfd(x,y) = \sup_{k \in \N} |\phi_k(x)-\phi_k(y)|, \quad \text{ for every } x,y \in X,\]
  and let $\EE \subset \Lipb(X,\sfd)$ be the smallest unital subalgebra of functions on $X$ containing $\{\phi_k\}_{k \in \N}$. Finally,  let $p,q \in (1,+\infty)$ be (not necessarily conjugate) fixed exponents and  let $\mm$ be a positive and finite Borel measure on $\prob_p(X, \sfd)$; then the algebra of cylinder functions generated by $\EE$, $\ccyl{\prob(X)}{\EE}$, is dense in $q$-energy in $D^{1,q}(\prob_p(X, \sfd), W_{p, \sfd}, \mm)$.
\end{theorem}
\begin{proof} Let us consider the map $\iota:X\to \ell^\infty(\N)$ defined as
\[ \iota(x)= (\phi_k(x))_k, \quad x \in X.\]
Then $\iota$ is an isometry between the metric spaces $(X,\sfd)$ and $(\iota(X), \|\cdot\|')$, where $\|\cdot\|'$ is the restriction to $\iota(X)$ of the norm $\|\cdot\|_\infty$ in $\ell^\infty(\N)$. This of course implies that $\mathcal{J}: \prob_p(X, \sfd) \to \prob_p(\iota(X), \sfd_{\|\cdot\|'})$ defined as
\[ \mathcal{J}(\mu):= \iota_\sharp \mu, \quad \mu \in \prob_p(X, \sfd),\]
is an isometry between the metric spaces $\prob_p(X, \sfd)$ and $\prob_p(\iota(X), \sfd_{\|\cdot\|'})$. If we set $\AA:=\mathcal{J}^*(\AA')$ and $\tilde{\mm}:=\mathcal{J}_\sharp \mm$, where 
\[ \mathcal{J}^*F = F \circ \mathcal{J},\quad F \in \AA',\]
then the spaces $H^{1,q}(\prob_p(X, \sfd), W_{p, \sfd}, \mm; \AA)$ and $H^{1,q}(\prob_p(\iota(X), \sfd_{\|\cdot\|'}), W_{p, \sfd_{\|\cdot\|'}}, \tilde{\mm}; \AA')$ are isomorphic (see \cite[Theorem 5.3.3]{Savare22}). Moreover we can see that $\AA \subset \EE$ so that it is enough to prove that the algebra $\AA'$ is dense in $q$-energy in $D^{1,q}(\prob_p(\iota(X), \sfd_{\|\cdot\|'}), W_{p, \sfd_{\|\cdot\|'}}, \tilde{\mm})$. To this aim, by Theorem \ref{theo:startingpoint}, it suffices to fix $\nu \in \prob_p(\iota(X), \sfd_{\|\cdot\|'})$ and prove that the function
  \begin{equation}
    \label{eq:98}
    F(\mu):=W_{p, \sfd_{\|\cdot\|'}}(\nu,\mu)\quad\text{satisfies}\quad
    |\rmD F|_{\star, q, \AA'}\le 1\quad\text{$\mm$-a.e.}.
  \end{equation}
  We split the proof in two steps.

  Step 1: it is sufficient to prove
  that, for every $h \in \N$, the function $F_h:\prob_p(\iota(X), \sfd_{\|\cdot\|'})\to\R$
  \begin{equation}
    \label{eq:99}
    F_h(\mu):=W_{p, \sfd_{\|\cdot\|_{\infty}}}(\hat\pi^h_\sharp\nu,\hat \pi^h_\sharp\mu)\quad
    \text{satisfies}\quad
    |\rmD F_h|_{\star, q, \AA'}\le 1\quad\text{$\mm$-a.e.},
  \end{equation}
  where $\hat\pi^h: \ell^{\infty}(\N) \to \ell^\infty(\N)$ is defined as
  \[ \hat\pi^h((x_k)_k):=(x_1, \dots, x_k, 0, \dots), \quad (x_k)_k \in \ell^{\infty}(\N).\]
  In fact, using the continuity property of the Wasserstein distance, it is
  clear that for every $\mu\in \prob_p(\iota(X), \sfd_{\|\cdot\|'})$
  \begin{equation}
    \label{eq:100} 
    \lim_{n\to\infty}F_h(\mu)=F(\mu), 
  \end{equation}
   so that it is enough to apply Theorem \ref{thm:omnibus}(1)-(3) to obtain \eqref{eq:98}. 
  
  Step 2: 
   Let $h \in \N$ be fixed and let us denote by $W_{p,h}$ the Wasserstein distance on $\prob_p(\R^h, \sfd_{\|\cdot\|_{h,\infty}})$, where $\|\cdot\|_{h, \infty}$ is the $\infty$ norm on $\R^h$. It is easy to check that
  \[ W_{p,h}(\pi^h_\sharp \mu_0, \pi^h_\sharp \mu_1) = W_{p, \sfd_{\|\cdot\|_\infty}}(\hat \pi^h_\sharp \mu_0, \hat \pi^h_\sharp \mu_1) \quad \text{ for every } \mu_0, \mu_1 \in \prob_p(\iota(X), \sfd_{\|\cdot\|'}).\]
  Thus, if we define the function $\tilde F_h:\prob_p(\R^h, \sfd_{\|\cdot\|_{h, \infty}})\to\R$
  as
  \[\tilde F_h(\mu):=W_{p,h}(\pi^h_\sharp\nu,\mu)\]
  we get that 
 \[ F_h(\mu)=\tilde F_h(\pi^h_\sharp \mu).\]  

  We also introduce the measure $\mm_h$ which
  is the push-forward of $\mm$ through the ($1$-Lipschitz) map
  $P^h: \prob_p(\iota(X), \sfd_{\|\cdot\|'})\to \prob_p(\R^h, \sfd_{\|\cdot\|_{h, \infty}})$ defined as $P^h(\mu):=\pi^h_\sharp \mu$.
  By  Theorem \ref{thm:main} applied to $H^{1,q}(\prob_p(\R^h, \sfd_{\|\cdot\|_{h, \infty}}),W_{p,\sfd_{\|\cdot\|_{h, \infty}}},\mm_h)$, 
  we can find a sequence of cylinder functions $\tilde F_{h,n}\in
  \ccyl{\prob(\R^h)}{\mathscr{P}_{h, \Phi(h)}}$, $n\in \N$, 
  such that
  \begin{gather}
    \label{eq:105}
    \tilde F_{h,n}\to \tilde F_h\text{ in
      $\mm_h$-measure},\\
    \label{eq:107}
    \lip_{W_{p,\sfd_{\|\cdot\|_{h, \infty}}}} \tilde F_{h,n}\to g_h\quad\text{in
    }L^q(\prob_p(\R^h, \sfd_{\|\cdot\|_{h, \infty}}),\mm_h)\quad\text{with }g_h\le 1
    \text{ $\mm_h$-a.e.}    
  \end{gather}
  We thus consider the functions $F_{h,n}\in \AA'$
  defined as in \eqref{eq:102} by
  \begin{equation}
    \label{eq:108}
    F_{h,n}(\mu):=\tilde F_{h,n}(\pi^h_\sharp\mu)=\tilde
    F_{h,n}(P^h(\mu))\quad
    \text{for every }\mu\in \prob(\iota(X)).
  \end{equation}
  Since for every $\eps>0$
    \begin{align*}
    \mm\Big(\big\{\mu:|F_{h,n}(\mu)-F_h(\mu)|>\eps\big\}\Big)&=
    \mm\Big(\big\{\mu:|\tilde F_{h,n}(P^h(\mu))-\tilde
    F_h(P^h(\mu))|>\eps\big\}\Big)\\&=
    \mm_h\Big(\big\{\mu:|\tilde F_{h,n}(\mu)-\tilde F_h(\mu)|>\eps\big\}\Big),    
  \end{align*}
  \eqref{eq:105} yields that $F_{h,n}\to F_h$ in $\mm$-measure as
  $n\to\infty$.

  On the other hand, \eqref{eq:104} and Remark \ref{rem:separable} yield
  \begin{align*}
    \lip_{W_{p, \sfd_{\|\cdot\|'}}} F_{h,n}(\mu)&\le
    \lip_{W_{p,h}} \tilde F_{h,n}(P^h(\mu))
  \end{align*}
  so that, up to a unrelabeled subsequence, we have
  \begin{displaymath}
    \lip_{W_{p, \sfd_{\|\cdot\|'}}} F_{h,n}\weakto G_h \le g_h\circ P^h\quad\text{in }L^q(\prob_p(\B, \sfd_{\|\cdot\|}),\mm)
  \end{displaymath}
  and $g_h\circ P^h\le 1$ $\mm$-a.e.~in $\prob_p(\iota(X), \sfd_{\|\cdot\|'})$.
   By Theorem \ref{thm:omnibus}(1)-(3), we obtain \eqref{eq:99}, concluding the proof.
\end{proof}
\begin{corollary}\label{cor:densban} Let $(\B, \|\cdot\|)$ be a separable Banach space and let $\mm$ be a positive and finite Borel measure on $\prob_p(\B, \sfd_{\|\cdot\|})$. Then the algebra $\ccyl{\prob(\B)}{\rmC_b^1(\B)}$ is dense in $q$-energy in $D^{1,q}(\prob_p(\B, \sfd_{\|\cdot\|}), W_{p, \sfd_{\|\cdot\|}}, \mm)$.
\end{corollary}
\begin{proof} It is enough to consider the functions of the form $\phi_k(x):= \la x_k^*,x \ra$, $x \in \B$, where $\{x_k^*\}_{k \in \N}$ is a subset of unit sphere in $\B^*$ separating the points in $\B$. Notice that each $\phi_k$ belongs to $\rmC_b^1(\B)$.
\end{proof}

\begin{corollary}\label{cor:densman} Let $(\M,g)$ be a complete Riemannian manifold and let $\mm$ be a positive and finite Borel measure on $\prob_p(\M, \sfd_{g})$, where $\sfd_g$ is the Riemannian distance induced by $g$. Then the algebra $\ccyl{\prob(\M)}{\rmC_c^\infty(\M)}$ is dense in $q$-energy in $D^{1,q}(\prob_p(\M, \sfd_{g}), W_{p, \sfd_{g}}, \mm)$.
\end{corollary}
\begin{proof} It is enough to consider a sequence of functions $(\phi_k)_k \subset \rmC_c^\infty(\M)$ satisfying \eqref{eq:compseq} for $\sfd_g$. 
\end{proof}

\section{Reflexivity, uniform convexity and Clarkson's inequalities}\label{sec:refl}
The aim of this section is to study how some properties of a Banach space (or its dual) pass to the Wasserstein Sobolev space built on it. In the whole section $(\B, \|\cdot\|)$ is a fixed separable Banach space and $\mm$ is a positive and finite Borel measure on $\prob_p(\B, \sfd_{\|\cdot\|})$.

To study the properties of the pre-Cheeger energy as in \eqref{eq:53} we will use the concept of $L^q$-direct integral of Banach spaces (\cite{grupre, rms})  (see also the related notion of Banach bundle \cite{pasqua}), that we introduce here briefly in the form that it is best suited to our needs. Let $\VV$ be the vector space of functions
\[ \VV := \left \{ f: \B \to \B^* : f \text{ is measurable and  } \int_{\B} \|f(x)\|_*^{p'} \d \mu(x) < +\infty \text{ for $\mm$-a.e. } \mu \in \prob_p(\B, \sfd_{\|\cdot\|}) \right \}. \]
It is not difficult to check that, for every $f \in \VV$, the map $M_f: \prob_p(\B, \sfd_{\|\cdot\|}) \to \R$ defined as
\[ M_f(\mu) := \|f\|_{L^{p'}(\B, \mu; (\B^*, \|\cdot\|_*))}, \quad \mu \in \prob_p(\B, \sfd_{\|\cdot\|}),\]
is measurable.
Let us denote by $\|\cdot\|_{*,p',\mu}$ the ${L^{p'}(\B, \mu; (\B^*, \|\cdot\|_*))}$-norm on the completion of $\VV_\mu:= \VV / \sim_{\mu}$, where $\sim_\mu$ is the equivalence relation of the equality $\mu$-a.e.. We say that $s: \prob_p(\B, \sfd_{\|\cdot\|}) \to \bigsqcup_{\mu} \VV_\mu$ (i.e.~$s(\mu) \in \VV_\mu$ for every $\mu \in \prob_p(\B, \sfd_{\|\cdot\|})$) is a simple function if there exist a finite measurable partition of $\prob_p(\B, \sfd_{\|\cdot\|})$, $\{A_k\}_{k=1}^N$, and values $\{f_k\}_{k=1}^N\subset \VV$ such that $s(\mu)=\sum_{k=1}^N \bambalau_{A_k}(\mu)f_k$ for every $\mu \in \prob_p(\B, \sfd_{\|\cdot\|})$. A function $G:\prob_p(\B, \sfd_{\|\cdot\|}) \to \bigsqcup_{\mu} \VV_\mu$ is said to be Bochner measurable if there exists a sequence of simple functions $(s_k)_k$ such that $\|G(\mu)-s_k(\mu)\|_{*,p',\mu} \to 0$ for $\mm$-a.e.~$\mu \in \prob_p(\B, \sfd_{\|\cdot\|})$. The direct integral of $(\VV_\mu)_{\mu}$ with respect to $\mm$ is the vector space of Bochner measurable functions $G:\prob_p(\B, \sfd_{\|\cdot\|}) \to \bigsqcup_{\mu} \VV_\mu$ modulo equivalence $\mm$-a.e.~and it is denoted by $\int_{\prob_p(\B, \sfd_{\|\cdot\|})}^{\bigoplus} \VV_\mu \d \mm(\mu)$. The $L^q$ direct integral $\left ( \int_{\prob_p(\B, \sfd_{\|\cdot\|})}^{\bigoplus} \VV_\mu \d \mm(\mu) \right )_{L^q}$ of $(\VV_\mu)_{\mu}$ with respect to $\mm$ is the subspace of $\int_{\prob_p(\B, \sfd_{\|\cdot\|})}^{\bigoplus} \VV_\mu \d \mm(\mu)$ consisting of those functions $G$ such that the (measurable) map $\mu \mapsto \|G(\mu)\|_{*,p',\mu}$ is in $L^q(\prob_p(\B, \sfd_{\|\cdot\|}), \mm)$.

\begin{proposition}\label{prop:diuc} Let $(\B, \|\cdot\|)$ be a separable Banach space. Then $\left ( \int_{\prob_p(\B, \sfd_{\|\cdot\|})}^{\bigoplus} \VV_\mu \d \mm(\mu) \right )_{L^q}$ is a  Banach space. If, in addition, $(\B, \|\cdot\|)$ is reflexive (resp.~$(\B^*, \|\cdot\|_*)$ is uniformly convex), then $\left ( \int_{\prob_p(\B, \sfd_{\|\cdot\|})}^{\bigoplus} \VV_\mu \d \mm(\mu) \right )_{L^q}$ is reflexive (resp.~uniformly convex). Finally, if $F \in \ccyl{\prob(\B)}{\rmC_b^1(\B)}$, then $\rmD_p F \in \left ( \int_{\prob_p(\B, \sfd_{\|\cdot\|})}^{\bigoplus} \VV_\mu \d \mm(\mu) \right )_{L^q}$.
\end{proposition}
\begin{proof} The completeness property  can be found in \cite[Proposition 3.2]{grupre} (see also \cite{rms} where the notion of direct integral was introduced for the first time). 
\quad \\
 If  $(\B^*, \|\cdot\|_*)$ is uniformly convex, then we can use \cite[Theorem 2.2]{preprint}, together with the fact that $L^q(\prob_p(\B, \sfd_{\|\cdot\|}), \mm)$ is uniformly convex and that the modulus of convexity of $\VV_\mu$ is larger than the one of $L^{p'}(\B,\mu; (\B^*, \|\cdot\|))$, being $\VV_\mu$ just a subspace of the latter, and the fact that the modulus of convexity of the Bochner space $L^{p'}(\B,\mu; (\B^*, \|\cdot\|))$ does not depend on $\mu$ but only on $p'$ and on the modulus of convexity of the uniformly convex space $(\B^*, \|\cdot\|_*)$.
\quad \\
If $\B$ is reflexive, then also $\B^*$ and thus $L^{p'}(\B, \mu; (\B^*, \|\cdot\|_*))$ is reflexive so that we can apply \cite[Theorem 6.19]{rms}, also noting that, obviously, $L^q(\prob_p(\B, \sfd_{\|\cdot\|}), \mm)$ is reflexive.
\quad \\
Let now $F=\psi \circ \lin{\pphi} \in \ccyl{\prob(\B)}{\rmC_b^1(\B)}$, with $\psi \in \rmC_b^1(\R^N)$ and $\pphi \in (\rmC_b^1(\B))^N$, for some $N \in \N$; notice that we are looking at $\rmD_p F$ as the map sending $\mu \in \prob_p(\B, \sfd_{\|\cdot\|})$ to $\rmD_p[\mu]: \B \to \B^*$. Of course $\rmD_p F[\mu] \in \VV_\mu$ for every $\mu \in \prob_p(\B, \sfd_{\|\cdot\|})$ and it is uniformly bounded. Moreover it is Bochner measurable since, in order to approximate it with simple functions, it is enough to approximate every term of the form $\partial \psi_n(\int_\B \phi_1 \d \mu, \dots, \int_\B \phi_N \d \mu)$ with $\R$-valued simple function $(s_k^n)_k$ on $\prob_p(\B, \sfd_{\|\cdot\|})$ (which is possible since this map is measurable) so that 
\[ s_k(\mu):= \sum_{n=1}^N s_k^n(\mu) \nabla \phi_n, \quad \mu \in \prob_p(\B, \sfd_{\|\cdot\|}),\]
is a sequence of simple functions approximating $\rmD_p F$. 
\end{proof}

\begin{theorem}\label{thm:refl} Let $(\B, \|\cdot\|)$ be a separable reflexive Banach space,  let $p,q \in (1,+\infty)$ be (not necessarily conjugate) fixed exponents  and let $\mm$ be a positive and finite Borel measure on $\prob_p(\B, \sfd_{\|\cdot\|})$. Then the Sobolev space $H^{1,q}(\prob_p(\B, \sfd_{\|\cdot\|}), W_{p, \sfd_{\|\cdot\|}}, \mm)$ is reflexive.
\end{theorem}
\begin{proof}
It is sufficient to provide a linear isometry $\iota$ from $H^{1,q}(\prob_p(\B, \sfd_{\|\cdot\|}), W_{p, \sfd_{\|\cdot\|}}, \mm)$ into a reflexive Banach space. For simplicity let us denote by 
\[ X:=L^q(\prob_p(\B, \sfd_{\|\cdot\|}), \mm), \quad Y:=\left ( \int_{\prob_p(\B, \sfd_{\|\cdot\|})}^{\bigoplus} \VV_\mu \d \mm(\mu) \right )_{L^q}\]
and let us define $\gG \subset X \times Y$ as the closure in $X \times Y$ (endowed with the norm $\|(x,y)\|_{X\times Y}^q = \|x\|_X^q+\|y\|_Y^q$) of
\[ \{ (F,\rmD_pF) : F \in \ccyl{\prob(\B)}{\rmC_b^1(\B)} \}.\]
Denoting by $\pi^X: X \times Y \to X$ the projection $\pi^X(x,y)=x$ for every $x \in X$, we define the sections of $\gG$ by
\[ \dD_\mm F := \{ G \in Y : (F,G) \in \gG \}, \quad F \in \pi^X(\gG). \]
It is clear that $\gG$ is closed and convex so that  it  is also weakly closed; let us show that $H^{1,q}(\prob_p(\B, \sfd_{\|\cdot\|}), W_{p, \sfd_{\|\cdot\|}}, \mm) \subset \pi^X(\gG)$: if $F \in H^{1,q}(\prob_p(\B, \sfd_{\|\cdot\|}), W_{p, \sfd_{\|\cdot\|}}, \mm)$, by Corollary \ref{cor:densban}, we can find a sequence $(F_n)_n \subset \ccyl{\prob(\B)}{\rmC_b^1(\B)}$ such that 
\[ F_n \to F, \quad \lip_{W_{p, \sfd_{\|\cdot\|}}} F_n \to |\rmD F|_{\star, q} \quad \text{ in } X.\]
By Proposition \ref{prop:equality}, we have that $\|\rmD_p F_n\|_Y$ is uniformly bounded so that, by the reflexivity of $Y$, we can find a subsequence of $(F_{n_k})_k$ such that
\[ (F_{n_k}, \rmD_p F_{n_k}) \weakto (F, G) \text{ in } X \times Y.\]
This gives that $F \in \pi^X(\gG)$ and that $\|G\|_Y  \le  \CE_q^{1/q}(F)$. Moreover, if we consider $F \in H^{1,q}(\prob_p(\B, \sfd_{\|\cdot\|}), W_{p, \sfd_{\|\cdot\|}}, \mm)$  and  $G \in \dD_{\mm}F$, then there exists $(F_n)_n \subset \ccyl{\prob(\B)}{\rmC_b^1(\B)}$ such that $(F_n, \rmD_p F) \to (F,G)$ in $X \times Y$ so that $\|G\|_Y = \lim_n \|\rmD_p F_n\|_Y= \lim_n \pCE_{q}^{1/q}(F_n) \ge \CE_q^{1/q}(F)$ so that 
\[ \min_{G \in \dD_\mm F} \|G\|_Y =\CE_q^{1/q}(F), \quad F \in H^{1,q}(\prob_p(\B, \sfd_{\|\cdot\|}), W_{p, \sfd_{\|\cdot\|}}, \mm). \]

We define $\gG_0:= \gG / (\{0\} \times \dD_\mm 0)$ with the quotient norm $\|\cdot\|_{Y,0}$: elements of $\gG_0$ are equivalence classes $[F] = \{ (F,G) : G \in \dD_\mm F \}$ and the quotient norm is simply given by 
\[ \|[F]\|_{Y,0} = \inf_{G \in \dD_\mm F} \left ( \|F\|^q_X + \|G\|_Y^q) \right )^{1/q} = \left ( \|F\|^q_X +  \min_{G \in \dD_\mm F} \|G\|_Y^q)  \right )^{1/q} = \|F\|_{H^{1,q}(\prob_p(\B, \sfd_{\|\cdot\|}), W_{p, \sfd_{\|\cdot\|}}, \mm)}.\]
The linear isometry is thus given by $\iota: H^{1,q}(\prob_p(\B, \sfd_{\|\cdot\|}), W_{p, \sfd_{\|\cdot\|}}, \mm) \to \gG_0$ defined as
\[ \iota(F) = [F], \quad F \in H^{1,q}(\prob_p(\B, \sfd_{\|\cdot\|}), W_{p, \sfd_{\|\cdot\|}}, \mm).\]
Since closed subspaces of reflexive spaces are reflexive and quotients of reflexive spaces are reflexive, $\gG_0$ is reflexive and this concludes the proof.
\end{proof}

The next result \cite[Lemma 2]{charuc} provides a quantitative version of the uniform convexity property that will pass from the pre-Cheeger energy to the Cheeger energy.
\begin{proposition}\label{prop:charuc} Let $(\WW, \|\cdot\|_{\WW})$ be a Banach space. Then $(\WW, \|\cdot\|_{\WW})$ is uniformly convex if and only if for every $t \in (1,+\infty)$ there exists a strictly increasing and continuous function $g_t:[0,2] \to [0,+\infty)$ such that 
\[ \left \| \frac{x+y}{2} \right \|_{\WW}^t + \left ( \|x\|_{\WW} \vee \|y\|_{\WW} \right )^t g_t \left ( \frac{\|x-y\|_{\WW}}{\|x\|_{\WW} \vee \|y\|_{\WW} )}\right ) \le \frac{1}{2}\|x\|_{\WW}^t + \frac{1}{2}\|y\|_{\WW}^t \quad \text{ for every } x,y \in \WW.\]
\end{proposition}

\begin{theorem}\label{thm:uc} Let $(\B, \|\cdot\|)$ be a separable Banach space such that $(\B^*, \|\cdot\|_*)$ is uniformly convex,  let $p,q \in (1,+\infty)$ be (not necessarily conjugate) fixed exponents  and let $\mm$ be a positive and finite Borel measure on $\prob_p(\B, \sfd_{\|\cdot\|})$. Then the Sobolev space $H^{1,q}(\prob_p(\B, \sfd_{\|\cdot\|}), W_{p, \sfd_{\|\cdot\|}}, \mm)$ is uniformly convex.
\end{theorem}
\begin{proof} By the uniform convexity of $\left ( \int_{\prob_p(\B, \sfd_{\|\cdot\|})}^{\bigoplus} \VV_\mu \d \mm(\mu) \right )_{L^q}$ provided by Proposition \ref{prop:diuc} and using Proposition \ref{prop:charuc}, we have that for every $t \in (1,+\infty)$ there exists a continuous and strictly increasing function $g_t:[0,2]\to [0,+\infty)$ such that 
\begin{align}
    \begin{split}\label{eq:disuguc}
    \pCE_q \left ( \frac{1}{2}(F+G) \right )^{t/q} &+ \left (\pCE_q^{1/q}(F) \vee \pCE_q^{1/q}(G) \right )^t g_t \left ( \frac{ \pCE_q^{1/q}(F-G)}{\pCE_q^{1/q}(F) \vee \pCE_q^{1/q}(G)} \right ) \\
    &\le \frac{1}{2}\pCE_q^{t/q}(F) + \frac{1}{2}\pCE_q^{t/q}(G) \quad \text{ for every } F,G \in \ccyl{\prob(\B)}{\rmC_b^1(\B)}, 
    \end{split}
\end{align}
where we are able to identify the norm in the direct integral and the $q$-th root of the pre-Cheeger energy thanks to Proposition \ref{prop:equality}.
If $F,G \in H^{1,q}(\prob_p(\B, \sfd_{\|\cdot\|}), W_{p, \sfd_{\|\cdot}}, \mm)$, thanks to Corollary \ref{cor:densban}, we can find sequences $(F_n)_n, (G_n)_n \subset \ccyl{\prob(\B)}{\rmC_b^1(\B)}$ such that
\[ F_n \to F, \quad G_n \to G \quad \text{ in } L^q(\prob_p(\B, \sfd_{\|\cdot\|}), \mm), \quad \pCE_q(F_n) \to \CE_q(F), \quad \pCE_q(G_n) \to \CE_q(G). \]
By the lower semicontinuity of the Cheeger energy w.r.t.~the $L^q$ convergence and the continuity and monotonicity of $g_t$, we obtain that \eqref{eq:disuguc} holds for $\CE_{q}$ instead of $\pCE_{q}$. By Proposition \ref{prop:charuc} we obtain that $\CE_q^{1/q}$ is uniformly convex, in the sense that, for every $\eps>0$ there exists a $\delta(\eps)>0$ such that, whenever $F,G \in H^{1,q}(\prob_p(\B, \sfd_{\|\cdot\|}), W_{p, \sfd_{\|\cdot\|}}, \mm)$ are such that $\CE_q^{1/q}(F)=\CE_q^{1/q}(G)=1$ and $\CE_q^{1/q}(F-G) \ge \eps$, then $\CE_q^{1/q}(F+G) < 2(1-\delta(\eps))$. Since the $L^q(\prob_p(\B, \sfd_{\|\cdot\|}, \mm)$-norm is uniformly convex, the $q$-sum (i.e.~the $q$-th root of the sum of the $q$-powers) of it with $\CE_q^{1/q}$ is uniformly convex (see e.g.~\cite[Theorem 1]{clarkson}). This means that the Sobolev norm is uniformly convex and concludes the proof.
\end{proof}

Given a separable Banach space $(\B, \|\cdot\|)$ and a positive and finite Borel measure $\mm$ on $\B$, we can consider the measure $\mathsf{D}_\sharp \mm$ on $\prob(\B, \sfd_{\|\cdot\|})$, where $\mathsf{D}: \B \to \prob(\B, \sfd_{\|\cdot\|})$ is defined as 
\[\mathsf{D}(x):= \delta_x, \quad x \in \B.\]
It is immediate to check that, for every $q \in (1,+\infty)$, the Sobolev spaces $H^{1,q}(\B, \sfd_{\|\cdot\|}, \mm)$ and $H^{1,q}(\prob_2(\B, \sfd_{\|\cdot\|}), W_{2, \sfd_{\|\cdot\|}}, \mathsf{D}_\sharp \mm)$ are isomorphic, see also \cite[Section 5.2]{FSS22}. Notice that the choice $p=2$ is irrelevant, since $\mathsf{D}(\B) \subset \prob_p(\B, \sfd_{\|\cdot\|})$ for every $p \in [1,+\infty]$. 
As a consequence of this observation and Theorems \ref{thm:refl} and \ref{thm:uc} we obtain the following corollary.

\begin{corollary}\label{cor:grazieref} Let $(\B, \|\cdot\|)$ be a separable Banach space, let $q \in (1,+\infty)$ be a fixed exponent, and let $\mm$ be a positive and finite Borel measure on $\B$. If $(\B, \|\cdot\|)$ is reflexive (resp.~if $(\B^*, \|\cdot\|_*)$ is uniformly convex), then $H^{1,q}(\B, \sfd_{\|\cdot\|}, \mm)$ is reflexive (resp.~uniformly convex). 
\end{corollary}

We remark that the statement regarding reflexivity in the above corollary was already known (see e.g.~\cite[Corollary 5.3.11]{Savare22}) while the part concerning uniform convexity is new, at least to our knowledge.

\subsection{Clarkson's type inequalities}\label{sec:clar}
In this subsection we show how one can deduce Clarkson's type inequalities for the Sobolev norm starting from analogous inequalities that are satisfied for the dual norm in the separable Banach space $\B$. The following inequality was introduced for the  first  time by \cite{boas} and it is a generalization of the classical Clarkson's inequalities \cite{clarkson}.

\begin{definition}\label{def:clark} Let $V$ be a real vector space and let $r,s \in (1,+\infty)$. We say that a functional $\jj: V \to [0,+\infty)$ satisfies the $(r,s)$-Boas inequality ($(r,s)$-(BI)) in $V$ if
\[ \left ( \jj(u+v)^r + \jj(u-v)^r \right )^{1/r} \le 2^{1/s'} \left (\jj(u)^s+\jj(v)^s \right )^{1/s} \quad \text{ for every } u,v \in V. \]
\end{definition}

\begin{lemma}\label{le:diolep} Let $V$ be a real vector space, let $q \in (1,+\infty)$ and let $\jj_1, \jj_2: V \to [0, +\infty)$ be two functionals satisfying the $(r,s)$-(BI) in $V$ for some $r,s \in (1,+\infty)$ such that $s \le q \le r$. Then the functional $\jj: V \to [0, +\infty)$ defined as
\[ \jj(u) := \left ( \jj_1(u)^q + \jj_2(u)^q \right )^{1/q}, \quad u \in V,\]
satisfies the $(r,s)$-(BI) in $V$ as well.
\end{lemma}
\begin{proof}
Let $u,v \in V$; let us set
\begin{align*}
&a_0:=\jj_1(u+v), \quad &&b_0:=\jj_1(u-v), \quad &&c_0:=\jj_2(u+v), \quad &&d_0:=\jj_2(u-v), \\
&a:=\jj_1(u), \quad &&b:=\jj_1(v), \quad &&c:=\jj_2(u), \quad &&d:=\jj_2(v).
\end{align*}
Since $\jj_1$ and $\jj_2$ satisfy the $(r,s)$-(BI) in $V$, we have
\begin{align}\label{eq:bos1}
    \left ( a_0^r + b_0^r \right )^{1/r} &\le 2^{1/s'} \left (a^s+b^s \right )^{1/s}, \\ \label{eq:bos2}
    \left ( c_0^r + d_0^r \right )^{1/r} &\le 2^{1/s'} \left (c^s+d^s \right )^{1/s}.
\end{align}
Then 
\begin{align*}
    2^{1/s'} \left (\jj(u)^s+\jj(v)^s \right )^{1/s} &= 2^{1/s'} \left ( (a^q+c^q)^{s/q} + (b^q+d^q)^{s/q} \right )^{1/s}\\
    &= 2^{1/s'} \left ( ((a^{s})^{q/s}+(c^s)^{q/s})^{s/q} + ((b^s)^{q/s}+(d^s)^{q/s})^{s/q} \right )^{1/s}\\
    & \ge 2^{1/s'} \left ( \left ( a^s+b^s \right )^{q/s} + \left ( c^s+d^s \right )^{q/s} \right )^{1/q} \\
    & \ge \left ( \left ( a_0^r+b_0^r\right )^{q/r}  + \left ( c_0^r + d_0^r\right )^{q/r} \right )^{1/q}\\
    & = \left ( \left ( (a_0^q)^{r/q}+(b_0^q)^{r/q}\right )^{q/r}  + \left ( (c_0^q)^{r/q} + (d_0^q)^{r/q}\right )^{q/r} \right )^{1/q} \\
    & \ge \left ( (a_0^q+c_0^q)^{r/q} + (b_0^q+d_0^q)^{r/q} \right )^{1/r} \\
    &= \left ( \jj(u+v)^r + \jj(u-v)^r \right )^{1/r},
\end{align*}
where we have used the triangular inequality for the $\alpha$-norm in $\R^2$ first with $\alpha=q/s$ (from the second to the third line) and then with $\alpha=r/q$ (from the fifth to the sixth line), and the inequalities \eqref{eq:bos1} and \eqref{eq:bos2} to pass from the third to the fourth line.
\end{proof}

\begin{proposition}\label{thm:hilb} Let $(X, \sfd, \mm)$ be a Polish metric-measure space, let $\AA$ be a separating unital subalgebra of
  $\Lip_b(X,\sfd)$ satisfying \eqref{eq:214-15} and let $q, r,s \in (1,+\infty)$. If $(\pCE_q)^{1/q}$ satisfies the $(r,s)$-(BI) in $\AA$, then $(\CE_q)^{1/q}$ satisfies the $(r,s)$-(BI) in $H^{1,q}(X,\sfd,\mm)$. If in addition $r' \le s \le q \le r$, then the Sobolev norm in $H^{1,q}(X,\sfd,\mm)$ satisfies the $(r,s)$-(BI).
\end{proposition}
\begin{proof} The proof that $(\CE_q)^{1/q}$ satisfies the $(r,s)$-(BI) in $H^{1,q}(X,\sfd,\mm)$ is the same of \cite[Theorem 2.15]{FSS22} for the case $r=s=q=2$. If in addition we know that $r' \le s \le q \le r$, then the $L^q(X, \mm)$-norm satisfies the $(r,s)$-(BI) in $H^{1,q}(X,\sfd,\mm)$ \cite[Theorem 1]{boas} so that we can apply Lemma \ref{le:diolep} and conclude that the Sobolev norm satisfies the $(r,s)$-(BI) in $H^{1,q}(X,\sfd,\mm)$.
\end{proof}

In the next Proposition \ref{prop:clarkpce} we show how to deduce the $(r,s)$-(BI) for the pre-Cheeger energy from the same inequality for the $L^{p'}(\B, \mu; (\B^*,\|\cdot\|_*))$ norm. Instead of stating our hypothesis in terms of the validity of a Clarkson's inequality in $(\B, \|\cdot\|)$, we state it directly in terms of the (BI) in $L^{p'}(\B, \mu; (\B^*,\|\cdot\|_*))$: the two notions are different in general, we refer to \cite{katsumoto} for the treatment of the relation between them.

\begin{proposition}\label{prop:clarkpce} Let $(\B, \|\cdot\|)$ be a separable Banach space and let $\mm$ be a positive and finite Borel measure on $\prob_p(\B, \sfd_{\|\cdot\|})$; suppose that there exists $(r,s)$ with $1 <s\le q \le r$ such that the $L^{p'}(\B, \mu; (\B^*,\|\cdot\|_*))$-norm satisfies the $(r,s)$-(BI) (cf.~Definition \ref{def:clark}) in $L^{p'}(\B, \mu; (\B^*,\|\cdot\|_*))$ for $\mm$-a.e.~$\mu \in \prob_p(\B, \sfd_{\|\cdot\|})$. Then the functional $\jj: \ccyl{\prob(\B)}{\rmC_b^1(\B)} \to [0,+\infty)$ defined as
\[ \jj(F):= (\pCE_{q}(F))^{1/q}, \quad F \in \ccyl{\prob(\B)}{\rmC_b^1(\B)},\]
satisfies the $(r,s)$-(BI) in $\ccyl{\prob(\B)}{\rmC_b^1(\B)}$.
\end{proposition}
\begin{proof}
Let $F, G \in \ccyl{\prob(\B)}{\rmC_b^1(\B)}$. We denote by $\|\cdot\|_{*,p', \mu}$ the $L^{p'}(\B, \mu; (\B^*,\|\cdot\|_*))$-norm for $\mu \in \prob_p(\B, \sfd_{\|\cdot\|})$ and we set
\[ U_\mu(x):= \rmD_p F(\mu, x), \quad V_\mu(x) := \rmD_p G(\mu, x), \quad (\mu, x) \in \prob_p(\B, \sfd_{\|\cdot\|}) \times \B.\]
 By hypothesis we have
\[ \left ( \|U_\mu+V_\mu\|_{*, p', \mu}^r + \|U_\mu-V_\mu\|_{*, p', \mu}^r \right )^{1/r} \le 2^{1/s'} \left (\|U_\mu\|_{*, p', \mu}^s+\|V_\mu\|_{*, p', \mu}^s \right )^{1/s} \]
for $\mm$-a.e.~$\mu \in \prob_p(\B, \sfd_{\|\cdot\|})$.
Computing the $L^q(\prob_p(\B, \sfd_{\|\cdot\|}), \mm)$-norm on both sides and applying Minkowski inequalities in $L^{q/r}(\prob_p(\B, \sfd_{\|\cdot\|}), \mm)$ and in $L^{q/s}(\prob_p(\B, \sfd_{\|\cdot\|}), \mm)$ (notice that $0 <q/r \le 1$ and $q/s \ge 1$), we get
\begin{align*}
\left ( \left ( \int_{\prob_p(\B, \sfd_{\|\cdot\|})} \|U_\mu+V_\mu\|_{*, p', \mu}^q \, \d \mm(\mu) \right )^{r/q} + \left ( \int_{\prob_p(\B, \sfd_{\|\cdot\|})} \|U_\mu-V_\mu\|_{*, p', \mu}^q \, \d \mm(\mu) \right )^{r/q} \right )^{1/r} \le \\
 2^{1/s'} \left ( \left ( \int_{\prob_p(\B, \sfd_{\|\cdot\|})} \|U_\mu\|_{*, p', \mu}^q \, \d \mm(\mu) \right )^{s/q} + \left ( \int_{\prob_p(\B, \sfd_{\|\cdot\|})} \|V_\mu\|_{*, p', \mu}^q \, \d \mm(\mu) \right )^{s/q} \right )^{1/s}
\end{align*}
i.e.
\[ \left ( \jj(F+G)^r + \jj(F-G)^r  \right )^{1/r} \le 2^{1/s'} \left ( \jj(F)^s + \jj(G)^s\right )^{1/s},\]
by Proposition \ref{prop:equality}.
\end{proof}

\begin{theorem}\label{cor:ilprimo} In the same hypotheses of Proposition \ref{prop:clarkpce}, assume in addition that $r'\le s$. Then the Sobolev norm in $H^{1,q}(\prob_p(\B, \sfd_{\|\cdot\|}),W_{p, \sfd_{\cdot}}, \mm)$ satisfies the $(r,s)$-(BI).
\end{theorem}
\begin{proof} This follows combining Proposition \ref{thm:hilb}, Corollary \ref{cor:densban} and Proposition \ref{prop:clarkpce}.
\end{proof}
\begin{corollary} If $\B=\R^d$ for some $d \in \N$ and $|\cdot|$ is the Euclidean norm, then the Sobolev norm in $H^{1,q}(\prob_p(\R^d, |\cdot|),W_{p, \sfd_{|\cdot|}}, \mm)$ satisfies the $(r,s)$-(BI) for every $1 < r' \le s \le q, p'\le r$.
\end{corollary}
\begin{proof} This follows by the fact that the $L^{p'}(\R^d, \mu; (\R^d, |\cdot|))$-norm satisfies the $(r,s)$-(BI) for every $1 < r' \le s \le p'\le r$ and Theorem \ref{cor:ilprimo}.
\end{proof}
\medskip
\paragraph{\em\bfseries Data availability statement}
Data sharing not applicable to this article as no datasets were generated or analysed during the current study.

\bibliographystyle{plain}
\bibliography{biblio}
\end{document}